\ifdefined\TransAMS
\documentclass[12pt]{tran-l}
\usepackage[margin=1in]{geometry}

\renewcommand{\paragraph}[1]{%
  \par\medskip
  \noindent\textbf{#1}.
}
\title[Implicit representations and prequantum structure]{Implicit representations of codimension-2 submanifolds and their prequantum structure}  
\author{Anonymous author(s)}

\else
\documentclass[12pt]{article}
\usepackage[margin=1in]{geometry}
\title{Implicit representations of codimension-2 submanifolds and their prequantum structure}
\makeatletter
\renewcommand\thefootnote{\textsuperscript{\@fnsymbol\c@footnote}}
\let\old@thanks\thanks 
\DeclareRobustCommand\thanks[2][]{
  \AddToHook{begindocument/end}{
    \if\relax#1\relax%
      \footnotemark%
    \else%
      \protect\refstepcounter{footnote}\protect\label{#1}%
    \fi%
    \protected@xdef\@thanks{%
      \@thanks\protect\footnotetext[\the\c@footnote]{#2}%
    }%
  }%
}
\AddToHook{begindocument/begin}{\let\thanks\old@thanks} 
\let\old@maketitle\maketitle
\def\maketitle{\old@maketitle\def\thefootnote{\@arabic\c@footnote}}
\makeatother
\author{Albert Chern\ref{a} and  Sadashige Ishida\ref{b}}
\thanks[a]{Email: alchern@ucsd.edu \\ \hspace*{1.5em} University of California San Diego, California, USA}
\thanks[b]{Corresponding author\\ \hspace*{1.5em} 
Email: sadashige.ishida@ist.ac.at \\ \hspace*{1.5em} Institute of Science and Technology Austria, Klosterneuburg, Austria}
\fi

\usepackage[en-US]{datetime2}

\usepackage{amssymb}
\usepackage{amsthm}
\usepackage{mathrsfs}

\usepackage{mathtools}

\mathtoolsset{showonlyrefs}

\usepackage{aliascnt}
\usepackage[colorlinks]{hyperref}

\hypersetup{
    colorlinks=true,
    linkcolor=black,
    filecolor=black,      
    urlcolor=black,
    citecolor=black
}

\newtheorem{theorem}{Theorem}[section]

\newcommand{\addtheorem}[2]{%
  \newaliascnt{#1}{theorem}%
  \newtheorem{#1}[#1]{#2}%
  \aliascntresetthe{#1}%
\expandafter\newcommand\csname #1autorefname\endcsname{#2}%

}

\addtheorem{corollary}{Corollary}
\addtheorem{lemma}{Lemma}
\addtheorem{proposition}{Proposition}
\addtheorem{conjecture}{Conjecture}

\theoremstyle{definition}
\addtheorem{definition}{Definition}
\addtheorem{assumption}{Assumption}
\addtheorem{example}{Example}
\addtheorem{remark}{Remark}

\numberwithin{equation}{section}

\usepackage{overpic}

\usepackage{mathabx}

\usepackage{nicematrix}

\usepackage{tikz-cd}

\usepackage{xcolor}

\usepackage{wrapfig}
\usepackage{graphicx} 

\usepackage{pb-diagram}

\definecolor{darkblue}{rgb}{0.0, 0.0, 0.55}

\usepackage[draft]{pdfcomment}

\usepackage[subrefformat=parens]{subcaption}
\definecolor{revcolor}{rgb}{0.0, 0.3, 0.8}

\newcommand{\on}[1]{\operatorname{#1}}

\def\cf{\emph{cf.}}
\def\ie{\emph{i.e.}}

\def\CC{\mathbb{C}}

\def\RR{\mathbb{R}}
\def\SS{\mathbb{S}}
\def\TT{\mathbb{T}}

\def\ZZ{\mathbb{Z}}

\def\cF{\mathcal{F}}

\def\cL{\mathcal{L}}

\def\cO{\mathcal{O}}
\def\cP{\mathcal{P}}

\def\cS{\mathcal{S}}

\def\bn{\mathbf{n}}

\def\bv{\mathbf{v}}

\def\bx{\mathbf{x}}

\def\Th{\Theta}

\def\Om{\Omega}

\newcommand{\id}{\mathrm{id}}

\newcommand{\Emb}{\mathrm{Emb}}

\newcommand{\UEmb}{\mathrm{UEmb}}

\usepackage{bbm}
\DeclareSymbolFont{bbold}{U}{bbold}{m}{n}
\DeclareSymbolFontAlphabet{\mathbbold}{bbold}
\newcommand{\ii}{\mkern1.5mu\mathbbold{i}\mkern1.5mu}

\renewcommand{\Im}{\operatorname{Im}}

\def\det{\operatorname{det}}

\def\im {\operatorname{im}}
\def\ker{\operatorname{ker}}

\def\id{\operatorname{id}}

\DeclareMathOperator{\Diff}{Diff}
\DeclareMathOperator{\SDiff}{SDiff}
\DeclareMathOperator{\Diffp}{Diff^+}

\DeclareMathOperator{\diff}{diff}
\DeclareMathOperator{\sdiff}{sdiff}

\let \DC \SemiD
\let \AlgDC \AlgSemiD

\DeclareMathOperator{\Gr}{\on{Gr}}

\newcommand{\Exp}{\on{Exp}}

\def\paramgam{{\tilde{\gamma}}}

\let\on=\operatorname

\usepackage[all,cmtip]{xy}
\begin{document}


\date{\today}












\maketitle


\begin{abstract}
This paper explores the geometry of the space of codimension-2 submanifolds. We implicitly represent these submanifolds by a class of complex-valued functions. 
We show that the space of all these implicit representations admit a prequantum bundle structure over the space of submanifolds, equipped with the well-known Marsden–Weinstein symplectic structure. 
This bundle allows a new geometric interpretation of the Marsden–Weinstein structure as the curvature of a connection form, which measures the average of volumes swept by the deformation of the $\SS^1$-family of hypersurfaces, defined as the phase level sets of the complex function implicitly representing a submanifold. 
\medskip


\medskip


\end{abstract}

\tableofcontents

\section{Introduction}
The shape space of closed (or exact) codiemnsion-2 submanifolds in a manifold \(M\) is a facsinating space. A primary example where such a space arises is hydrodynamics: co\-dimension-2 submanifolds represent singular vortices, namely point vortices in 2D and vortex filaments in 3D. 
With this example in mind, there is a natural symplectic structure defined for the space of closed codiemension-2 submanifolds in Hamiltonian hydrodynamics.  This is the so-called  \emph{Marsden--Weinstein} (MW) symplectic structure \cite{MarsdenWeinstein1983coadjoint,Khesin2012vortexMembrane}.
Although it was originally introduced in the context of fluid dynamics, the MW symplectic structure is purely geometric.  It is canonically defined whenever the ambient manifold \(M\) is equipped with a volume form.  The MW structure plays crucial roles beyond fluid dynamics, giving rise to other dynamical systems and integrable systems of space curves \cite{HallerVizman2003NonlinearGA,chern_knoeppel_pedit_pinkall_2020}.  


When it comes to a symplectic manifold \((\Sigma,\sigma)\) (\ie\@ a manifold \(\Sigma\) endowed with a closed non-degenerate 2-form \(\sigma\)), an interesting question is whether the closed form \(\sigma\) is also exact; that is, whether \(\sigma = d\vartheta\) for some \emph{symplectic potential} 1-form \(\vartheta\).  For example, the symplectic structure of the cotangent bundle \(\Sigma = \{(q,p)\in T^*Q\}\) of any manifold \(Q\) is defined by \(\sigma = d\vartheta\), where \(\vartheta = \langle p|dq\rangle\), and therefore it is exact.  For non-example, the area form \(\sigma\) on a sphere \(\Sigma = \SS^2\) is a non-exact symplectic form.  
A more relaxed condition for the exactness is that the symplectic form is the \emph{curvature} of some \emph{connection} of a \emph{circle bundle}. This is known as the following \emph{prequantum structure}.

\begin{definition}
    A \emph{prequantum bundle} \((P,\vartheta)\)  over a symplectic manifold \((\Sigma,\sigma)\) consists of a principal \(\SS^1\)-bundle \(\pi\colon P\to \Sigma\) and a connection form \(\vartheta\in\Omega^1(P)\) so that \(\pi^*\sigma = d\vartheta\).
\end{definition}

For example, the \emph{Hopf fibration} \(\SS^3\to\SS^2\) with a suitable 1-form \(\vartheta\) yields a prequantum bundle for the non-exact area form on \(\SS^2\).  All symplectic manifolds with symplectic potential 
admit a prequantum structure.

Note that the connection \(\vartheta\), which is classically defined as a Lie algebra-valued 1-form, is here represented by a real-valued 1-form.  This is because the structural group \(\SS^1\cong U(1)\) is one-dimensional with a Lie algebra identified with \(\RR\).  Following this note, one may generalize the notion of a prequantum bundle by replacing \(\SS^1\)-bundles more general principal \(G\)-bundles as long as the Lie group \(G\) is Ablian and one-dimensional.    

\begin{definition}\label{def:GeneralizedPrequantumBundle}
    A \emph{generalized prequantum bundle} \((P,\vartheta)\) over a symplectic manifold \((\Sigma,\sigma)\) is a principal \(G\)-bundle \(\pi\colon P\to\Sigma\) with structural group \(G = \SS^1\times H\) for some discrete Abelian group \(H\), together with a connection form \(\vartheta\in\Omega^1(P)\) so that \(\pi^*\sigma = d\vartheta\).
\end{definition}

Constructing a prequantum bundle of a symplectic manifold is the first step in \emph{geometric quantization}, a recipe that turns a classical mechanical system with phase space \((\Sigma,\sigma)\) into a quantum mechanical system for which the wavefunctions are sections of the complex line bundle associated with \(P\).  Having a prequantum structure is important not only for quantization.  The mere fact that a prequantum structure realizes the symplectic form as a curvature brings geometric interpretations to the values of the symplectic form.  For example, they correspond to the geometric phases or holonomies of some geometric process that represents parallel transports with respect to the connection.

Let us return to the Marsden--Weinstein symplectic space of codimension-2 submanifolds.  Is the MW form exact?  If not, does it have a prequantum structure?  Is it the geometric phase of a geometric process?

It is known that the MW form is exact when $\dim M =3$ and the volume form of $M$ is exact \cite{Tabachnikov2017bicycle,Padilla2019bubbleRings, brylinski2009loop}. We extend this result to arbitrary dimension (\autoref{thm:SymplecticPotentialForExactVolumeForm}). However, the volume form is no longer exact when \(M\) is a closed manifold.  

For a closed \(M\), Haller and Vizman \cite{HallerVizman2003NonlinearGA} show that the MW structure is prequantizable if and only if the total volume of \(M\) is an integer.
Their construction of a prequantum bundle is abstractly gluing together local symplectic potentials on a collection of open sets of the shape space. Such a construction does not directly tell us a geometric picture of how the MW form arises as the geometric phase of a concrete geometric process.
A more recent work \cite{Diez2020InducedDC} shows the existence of prequantum circle bundles for MW form with additional control over global monodromy data or differential characters.  The result is also based on general topological machinery that generates isomorphism classes of circle bundles given curvature and monodromy, rather than an explicit construction of the prequantum bundle.

\subsection{Main Results}\label{sec:MainResults}
In this article, we present an explicit construction of a prequantum bundle for the MW symplectic structure with an intuitive geometric interpretation.

A key step is to represent codimension-2 submanifolds \emph{implicitly}.  Similar to the implicit representation of codimension-1 hypersurfaces by level set functions, codimension-2 submanifolds can be expressed as the zero sets of complex-valued functions.
This implicit representation has been utilized in computing vortex filaments \cite{weissmann2014smoke}, vortex dynamics \cite{iwc2022implicit_filaments}, curve-shortening flow \cite{MBO2001diffusion}, and quantum vortex filaments governed by the Gross--Pitaevskii equation \cite{ogawa2002study,villois2016GrossPitaevskii,jerrard2018leapfrogging}.
Note that such implicit representation restricts the class of codimension-2 submanifolds to those that are homologically exact, meaning that they must bound codimension-1 hypersurfaces.  Nevertheless, exact codimension-2 submanifolds cover all cases of singular vortices in fluid dynamics, which are the main subjects of study with the Marsden--Weinstein structure.  In order to employ implicit representations, we focus only on exact codimension-2 submanifolds.  

Let \(M\) be an oriented manifold equipped with a volume form \(\mu\), which one may normalize so that \(\int_M\mu = 1\). Let \(\cO\) be the space of exact codimension-2 submanifolds in \(M\), represented by embeddings modulo reparameterizations.  We call \(\cO\) the \emph{explicit shape space},%
\footnote{More precisely, \(\cO\) is a connected component of the space of exact codimension-2 submanifolds. }
which is endowed with the MW sympelctic form \(\omega^{\rm MW}\in\Omega^2(\cO)\).
The implicit representation of each element in \(\cO\) is not unique, as multiple complex-valued functions can share the same zero set. This non-uniqueness makes the \emph{implicit shape space} \(\cF_\cO\) a fiber bundle over the explicit shape space \(\cO\).
%
\footnote{Similar to \(\cO\), the space \(\cF_\cO\) is a subset of the space of complex functions, representing elements of $\cO$.}

Each element \(\psi\in\cF_\cO\) is a complex-valued function over \(M\), while each element \(\gamma\in\cO\) is a codimension-2 submanifold in \(M\). The projection from \(\cF_\cO\) to \(\cO\) is to extract the zero set.

This bundle has a natural geometric interpretation.
Each complex function $\psi\in\cF_\cO$ representing a codimension-2 submanifold $\gamma\in\cO$ carries additional information in the form of its complex phase. The level sets of this phase function constitute an $\SS^1$-family of hypersurfaces in \(M\) that are bordered by $\gamma$.
Any motion of this configuration leads to these hypersurfaces moving across the space, from which one can measure the average volume swept out by these hypersurfaces.
In each fiber of \(\cF_\cO\), we define two hypersurface configurations as equivalent (\(\sim\)) if they can be continuously deformed into each other while keeping their boundaries \(\gamma\) fixed and maintaining zero net volume change throughout the motion.  
The space of equivalence classes is denoted by \(\cP\coloneqq\cF_\cO/\sim\), which is still a fiber bundle over \(\cO\).
We call \(\cP\) the \emph{volume bundle} for the implicit shape space.

It turns out that the volume bundle $\cP$ over $\cO$ is a \emph{generalized prequantum bundle} (\autoref{def:GeneralizedPrequantumBundle}) with each fiber being \(\SS^1\times H^1(M; \ZZ)\),
where the latter is the first cohomoology of $M$.
The connection form $\Th_\cP$ on this bundle is defined such that the horizontal lift of any path in $\cO$ corresponds to a motion of phase hypersurfaces that sweep out zero net volume along the path. With this setup, we have the following theorem.

\begin{theorem}[\autoref{th:generalized_prequantum_bundle} and \autoref{cor:prequantum_circle_bundle}]\label{th:informal:prequantum_circle_bundle}
    The fibration $\Pi_\cP\colon(\cP,\Th_\cP)\to (\cO,\omega^{\rm MW})$ is a generalized prequantum bundle
    with the structure group $G=\SS^1\times H^1(M;\ZZ)$.  
    In particular, the volumne bundle is a prequantum circle bundle for simply-connected \(M\).
\end{theorem}


In other words, the curvature of the connection \(\Theta_{\cP}\) agrees with the MW symplectic form.
As a consequence, this prequantum bundle allows the following geometric interpretation of the MW form:


\begin{corollary}[\autoref{cor:average_swept_volume}]\label{cor:intro_average_swept_volume}
Consider a closed path \(\partial D\) in \(\cO\) that bounds a 2-dimensional disk \(D\), representing a periodic motion of a codimension-2 submanifold \(\gamma_t \subset M\) for \(0 \leq t \leq 1\), with \(\gamma_0 = \gamma_1\). Let $[\psi_t]_\cP$ be a horizontal lift in \(\cP\) of $\gamma_t$ and  $\phi_t\coloneqq \psi_t/|\psi_t|\colon M\setminus \psi^{-1}_t(0) \to \SS^1$ be the phase map of a representative $\psi_t\in [\psi_t]_\cP$. Then \(\gamma_t\) bounds a family of hypersurfaces \(\{\sigma_t^s\}_{s\in \SS^1}\), defined by $\sigma_t^s = \phi_t^{-1}(s)$.
Assume that the average volume swept out by \(\sigma_t^s\) remains zero at each \(t\) \ie,
\begin{align}
    \int_{\SS^1}\int_{\sigma_t^s}\iota_{\partial_t{\sigma}_t^s} \mu \ ds = 0, 
\end{align}
where \(\mu\) is the volume form on \(M\).
Then, the volume enclosed between \(\sigma_0^s\) and \(\sigma_1^s\), averaged over \(s\in \SS^1\), equals to \(\iint_{D} \omega^{\rm MW}\).
\end{corollary}

In the limiting case of \autoref{cor:intro_average_swept_volume}, where the phase of $\psi_t$ becomes constant except a \(2\pi\) jump at on a single hypersurface \(\sigma_t\) bounding \(\gamma_t\), the result simplifies to the following corollary. This version has no explicit reference to the complex function $\psi_t$.

\begin{corollary}[\autoref{cor:single_surface_swept_volume}]
Let \(\{\gamma_t\}_{t \in [0,1]}\) be a path along \(\partial D\)  for some 2-dimensional disk \(D\) in $\cO$ as in \autoref{cor:intro_average_swept_volume}. Suppose that each \(\gamma_t\) bounds a hypersurface, i.e., \(\gamma_t = \partial \sigma_t\), and that the volume swept out by \(\sigma_t\) remains zero at each \(t\).
Then \(\iint_{D} \omega^{\rm MW}\) equals the volume enclosed between \(\sigma_0\) and \(\sigma_1\).
\end{corollary}

\ifdefined\TransAMS
\else
\paragraph{Acknowledgement}
We thank Chris Wojtan for his continuous support to the project through a number of discussions. We also thank Ioana Ciuclea, Bas Janssens, and Cornelia Vizman for discussions on prequantum bundles over nonlinear Grassmannians. The second author warmly acknowledges the hospitality of the University of California, San Diego, where part of this research was carried out during his visit. 

This project was funded in part by the European Research Council (ERC Consolidator Grant 101045083 CoDiNA) and the National Science Foundation CAREER Award 2239062. Some figures in the article were generated by the software Houdini and its education license was provided by SideFX.
\fi
\section{Preliminary}
We review the shape spaces of submanifolds and the Marsden--Weinstein (MW) structure on the codimension-2 shape space.  

Throughout, we let \(M\) be an \(m\)-dimensional oriented manifold equipped with a volume form \(\mu\in\Omega^m(M)\). 
Let \(\Diff(M)\) be the diffeomorphism group on \(M\), and \(\SDiff(M)\) be the volume-preserving diffeomorphism group on \(M\).  Let \(\Diff_0(M)\subseteq\Diff(M)\) denote the connected component that contains the identity.

\subsection{Space of Unparameterized Embeddings}
Let \(S\) be a \(k\)-dimensional compact, oriented manifold, with \(0\leq k\leq m\).  In the case of \(k = 0\), \(S\) is a finite set of oriented points. The space of smooth embeddings of \(S\) in \(M\),
\begin{align}
    \on{Emb}(S; M) \coloneqq \{ \paramgam \in C^\infty(S; M) \mid \on{rank}(d\paramgam) = k,\ \paramgam(s) = \paramgam(s') \implies s = s' \},
\end{align}
is an infinite dimensional manifold with the \(C^\infty\) Fr\'echet topology \cite{bauerBruverisMichor2014overview, Michor2019ManifoldsOM}.
Its tangent space at each \( \paramgam \in \on{Emb}(S, M) \) is given by the space of sections of the pullback tangent bundle:
\(
    T_\paramgam \on{Emb}(S, M) \cong \Gamma(\paramgam^* TM).
\)
That is, a tangent vector \( \dot\paramgam \in T_\paramgam \on{Emb}(S, M) \) assigns to each point \( s \in S \) a vector \( \dot\paramgam(s) \in T_{\paramgam(s)}M \). 

Each reparameterization of an embedding \(\tilde\gamma\in\on{Emb}(S;M)\) is a right action \(\tilde\gamma\mapsto\tilde\gamma\circ\varphi\) by an orientation-preserving diffeomorphism \(\varphi\in\Diffp(S)\).  
The \emph{shape space} of \emph{unparameterized embeddings} of \(S\) in \(M\) is the quotient of the embedding space modulo these actions:
\begin{align}
    \UEmb(S;M) \coloneqq \on{Emb}(S;M)/\Diffp(S).
\end{align}
The space \(\UEmb(S;M)\) is also referred to as the \emph{nonlinear Grassmannian of type \(S\)} \cite{HallerVizman2003NonlinearGA}, which is an infinite-dimensional manifold \cite{bauerBruverisMichor2014overview}, with proper tangent spaces. 
In the following, we denote an element of \(\UEmb(S;M)\) by \(\gamma\), and any representative of \(\gamma\) by \(\tilde\gamma\in \Emb(S;M)\).
The tangent space \(T_\gamma \UEmb(S;M)\) of \(\UEmb(S;M)\) at \(\gamma\) is the image \(d\pi|_{\tilde\gamma}(T_{\tilde\gamma}\Emb(S;M))\), where $\pi\colon  
    \Emb(S;M) \to \UEmb(S;M)$ is the quotient map.
In other words, \(T_\gamma\UEmb(S;M)\cong T_{\tilde\gamma}\Emb(S;M)/\ker d\pi|_{\tilde\gamma}\).  Here, \(\ker d\pi|_{\tilde\gamma}\) is given by infinitesimal right actions by the Lie algebra \(\diff(S)\) of \(\Diffp(S)\), implying \(\ker d\pi|_{\tilde\gamma} = d\tilde\gamma(\diff(S))\).  In particular, \(\ker d\pi|_{\tilde\gamma}\) consists of vector fields \(\dot{\tilde\gamma}\in\Gamma(\tilde\gamma^*TM)\) that are tangential to \(\gamma\), which do not change the shape of \(\gamma\).  

The shape space of all unparameterized \(k\)-dimensional surfaces \(S\) in \(M\) is the \(k\)-th \emph{nonlinear Grassmannian}
\begin{align}
    \on{Gr}_k(M)\coloneqq \bigsqcup_S \UEmb(S;M),
\end{align}
where the disjoint union iterates over \(k\)-dimensional compact oriented manifolds \(S\) up to diffeomorphisms.
The boundary operator \(\partial\) is well-defined for nonlinear Grassmannians \(\partial\colon\on{Gr}_k(M)\to\on{Gr}_{k-1}(M)\), where the boundary \(\partial\gamma\) of a shape \(\gamma\in\UEmb(S;M)\) is \(\partial\gamma = \pi(\tilde\gamma|_{\partial S})\), which is the restriction of any of its parameterization \(\tilde\gamma\in\Emb(S;M)\) on the boundary \(\partial S\), followed by the projection \(\pi\colon\Emb(S;M)\to\UEmb(S;M)\).  Define the shape spaces of \emph{closed} and \emph{exact} unparameterized \(k\)-surfaces respectively as
\begin{align}
    \on{Gr}_k^{\rm cl}(M)\coloneqq\bigsqcup_{S\colon\partial S = \varnothing}\UEmb(S;M) \subset\on{Gr}_k(M),\quad\on{Gr}_k^{\rm ex}(M)\coloneqq \partial\on{Gr}_{k+1}(M)\subset\on{Gr}_k^{\rm cl}(M).
\end{align}


\subsubsection{Diffeomorphism Action}
There is a left action of the diffeomorphism group $\Diff(M)$ of the ambient space \(M\) on the shape space $\on{Gr}_k(M)$
representing \emph{moving shapes in \(M\)}, defined by
\begin{align}\label{eq:MovingShapesInM}
    f \circ \gamma \coloneqq \pi(f \circ \tilde\gamma), \quad f\in\Diff(M),\quad \gamma\in \on{Gr}_k(M),
\end{align}
 where \(\paramgam \in \pi^{-1}\gamma\) is any parametrization of \(\gamma\).

Each connected component of $\on{Gr}_k(M)$ (respectively $\on{Gr}^{\rm cl}_k(M)$ and $\on{Gr}^{\rm ex}_k(M)$) is an orbit of the action of $\Diff_0(M)$, the connected component of $\Diff(M)$ containing $\id_M$. This is a classical result due to Thom (see, e.g., \cite{hirsch2012differential}). Consequently, any tangent vector $\dot\gamma \in T_\gamma \on{Gr}_k(M)$ can be written as
$\dot \gamma = u \circ \gamma \coloneqq d\pi|_{\paramgam}(u \circ \paramgam)$
for some $u \in \diff(M)\cong\Gamma(TM)$. 
Based on this connectivity argument, most of the geometric studies of \(\on{Gr}_k(M)\) can be restricted to a single \(\Diff_0(M)\)-orbit \(\cO\) in \(\on{Gr}_k(M)\).

In fact, for \(k<m-1\), the action by the subgroup $\SDiff_0(M)\subset \Diff_0(M)$ of volume-preserving diffeomorphisms is transitive on each \(\Diff_0(M)\)-orbit in \(\on{Gr}_k(M)\) \cite[Proposition 2]{HallerVizman2003NonlinearGA}.  In other words, \(\cO\) is also an \(\SDiff_0(M)\)-orbit.  Much of the theory involving moving shapes of codimension higher than 1 remains the same after restricting the diffeomorphism group \(\Diff_0(M)\) to \(\SDiff_0(M)\) and the space \(\diff(M)\) of vector fields to the space \(\sdiff(M)\) of divergence-free vector fields.

\subsubsection{Fundamental Vector Fields}
The group action of \(\Diff_0(M)\) on \(\cO\subset\on{Gr}_k(M)\) induces an infinitesimal action.
\begin{definition}
\label{def:FundamentalVectorFieldX}
    Let \(X^{(\cdot)}\colon \Gamma(TM)\to\Gamma(T\cO)\) be the map defined by
    \begin{align}
        (X^u)_{\gamma}\coloneqq (u\circ\gamma) \in T_\gamma\cO,\quad\text{for each \(u\in\Gamma(TM)\) and \(\gamma\in\cO\)}.
    \end{align}
    The vector field \(X^u\in\Gamma(T\cO)\) is called the \emph{fundamental vector field} corresponding to the generator \(u\in\Gamma(TM)\cong\diff(M)\).
\end{definition}
Note that \((X^u)_\gamma\) is the derivative of the group action \(\Diff_0(M)\times\cO\to\cO\), \((f,\gamma)\mapsto f\circ\gamma\), with respect to \(f\in\Diff_0(M)\) at the identity, evaluated on \(u\in \diff(M) = T_{\id_M}\Diff_0(M)\).
\begin{proposition}\label{prop:FundamentalVectorFieldIsLieAlgebraHomomorphism}
    \(X^{(\cdot)}\) is a Lie algebra homomorphism of vector fields.  That is, 
    \begin{align}
        [X^u,X^v] = X^{[u,v]}
    \end{align}
    where \([\cdot,\cdot]\) denotes the vector field Lie bracket.
\end{proposition}
\begin{proof}
    The left group action \((f,\gamma)\mapsto f\circ\gamma\), viewed as the map \(f\mapsto(\gamma\mapsto f\circ\gamma)\), defines a group homomorphism \(\Diff_0(M)\to \Diff(\cO)\).  The differential of this map at \(\id_M\in\Diff_0(M)\) is a Lie algebra homomorphism, given by the map \(X^{(\cdot)}\colon \diff(M)\to\diff(\cO)\). 
    Now apply the identifications \(\diff(M)\cong\Gamma(TM)\) and \(\diff(\cO)\cong\Gamma(T\cO)\).  Note, however, that the isomorphism \(\diff(M)\xrightarrow{\cong}\Gamma(TM)\) is a Lie algebra anti-homomorphism; that is, the Lie bracket \([\cdot,\cdot]_{\diff(M)}\) arising from the group structure and the vector field Lie bracket \([\cdot,\cdot]_{\Gamma(TM)}\) have opposite signs: \([u,v]_{\diff(M)} = -[u,v]_{\Gamma(TM)}\).  
    Nevertheless, applying this Lie algebra anti-isomorphism at both ends of the homomorphism  \(X\colon\diff(M)\to\diff(\cO)\) yields a Lie algerba homomorphism \(X\colon\Gamma(TM)\to\Gamma(T\cO)\).
\end{proof}

Recall that any tangent vector \(\dot\gamma\in T_\gamma\cO\) can be written as \(\dot\gamma = u\circ\gamma\) for some \(u\in\Gamma(TM)\), which is the evaluation \((X^u)_\gamma\) of \(X^u\in \Gamma(T\cO)\) at \(\gamma\).  Therefore, every tangent vector of \(\cO\) can be extended to a fundamental vector field.  This gives a convenient representation of each tangent space of the nonlinear Grassmannian:
\begin{align}\label{eq:XRepresentationOfTO}
    T_\gamma\cO = \{X^u_\gamma\,|\, u\in\Gamma(TM)\}.
\end{align}

\subsubsection{Integrals on Submanifolds}
The integral \(\int_\gamma\alpha = \int_{\tilde\gamma}\alpha\) of a differential form \(\alpha\in\Omega^k(M)\) is independent of the choice of parameterization \(\tilde\gamma\)  of \(\gamma\).  Therefore, each smooth differential form \(\alpha\in\Omega^k(M)\) defines a smooth function \(\gamma\mapsto \int_\gamma\alpha\) on \(\on{Gr}_k(M)\).  
The derivative of this evaluation function \(\int_{(\cdot)}\alpha\colon \on{Gr}_k(M)\to\RR\) along each tangent vector \(X_\gamma^u\in T_\gamma\on{Gr}_k(M)\), \(u\in\Gamma(TM)\), is given by the Leibniz integral rule 
\begin{align}
\label{eq:ReynoldsTransportTheorem}
    X_\gamma^u \left(\int_{(\cdot)}\alpha\right) = \int_{\gamma}\cL_u\alpha,
\end{align}
where \(\cL_u\) denotes the Lie derivative.  This follows from setting a time-dependent \(\gamma_t\) with \(\gamma_0 = \gamma\), \({\partial\over\partial t}|_{t=0}\gamma_t = X^u_\gamma = u\circ\gamma\) (\autoref{def:FundamentalVectorFieldX}), and taking time derivative of \(\int_{\gamma_t}\alpha\).

\subsection{Marsden--Weinstein Sympelctic Structure}
Let \(\cO\) be a \(\Diff_0(M)\)-orbit of the space \(\on{Gr}^{\rm ex}_{m-2}(M)\) of exact codimension-2 submanifolds.  The space \(\cO\) is a \emph{weak symplectic manifold}.  That is, it is equipped with a closed and \emph{weakly non-degenerate} 2-form.  This 2-form is the following \emph{Marsden--Weinstein} (MW) \emph{form} \(\omega^{\rm MW}\in\Omega^2(\cO)\).
Here we use the representation \eqref{eq:XRepresentationOfTO} by fundamental vector fields for tangent vectors in \(T_\gamma\cO\).

\begin{definition}
    On an unbounded manifold \(M\) with infinite volume \(\int_M\mu = \infty\), the MW form is defined by
\begin{align}\label{eq:UnboundedMWForm}
    \omega^{\rm MW}|_{\gamma}(X^u_\gamma,X^v_\gamma) \coloneqq \int_\gamma\iota_v\iota_u\mu,\quad\text{for  \(u,v\in\Gamma(TM)\)}.
\end{align}
On a closed manifold \(M\) with volume \(|M|\coloneqq\int_M\mu <\infty\), the MW form includes a normalizing factor%
\footnote{Many authors do not include this normalization. However, there are topological results about \(\omega^{\rm MW}\), such as its prequantizatbility \cite{HallerVizman2003NonlinearGA}, depend on conditions such as \(\int_M\mu\in\ZZ\).  This means that some topological results about \(\omega^{\rm MW}|_{\gamma}(\dot\gamma,\mathring\gamma) = \int_\gamma\iota_v\iota_u\mu\) are not scale invariant.  Our definition \eqref{eq:NormalizedMWForm} with an \(1\over |M|\) factor ensures the scale invariance.  The normalization can also be interpreted as that the volume form \(\mu\) is replaced by \(\overline\mu = {\mu\over|M|}\), which always satisfies \(\int_M\overline\mu = 1\) (\cf\@ the setup in \autoref{sec:MainResults}).}
\begin{align}
\label{eq:NormalizedMWForm}
     \omega^{\rm MW}|_{\gamma}(X^u_\gamma,X^v_\gamma) \coloneqq {1\over |M|}\int_\gamma\iota_v\iota_u\mu,
     \quad\text{for  \(u,v\in\Gamma(TM)\)}.
\end{align}
\end{definition}

The MW form \(\omega^{\rm MW}\) is weakly-nondegenerate in the sense that the map \(X^u_\gamma\mapsto \omega^{\rm MW}|_\gamma(X^u_\gamma,\cdot)\) from \(T_\gamma\cO\) to \(T_\gamma^*\cO\) is injective, but not surjective.  Throughout this article, we refer to such weak symplectic forms simply as \emph{symplectic}.  For more backgrounds on weak symplectic geometry, see \cite[Chapter VI]{Michor1984convenient} and \cite[Appendix B]{bauerMichorIshida2026symplectic}.

\subsubsection{Symplectic Potential (for unbounded \(M\))}\label{sec:explicit symplectic potential}
On a symplectic manifold, a \emph{symplectic potential} is a 1-form whose exterior derivative equals to the symplectic form. 

We show that the MW form admits a symplectic potential provided that the volume form \(\mu\) is exact: \(\mu = d\nu\) for some \(\nu\in\Omega^{m-1}(M)\).
Note that \(\mu\) is never exact on a closed \(M\) since it generates the fundamental class.  Hence \(M\) is not a closed manifold in this setting, and the MW form is taken to be \eqref{eq:UnboundedMWForm}.
\begin{theorem}\label{thm:SymplecticPotentialForExactVolumeForm}
    Suppose \(\nu\in\Omega^{m-1}(M)\) satisfies \(d\nu = \mu\).  Then \(\eta\in\Omega^{1}(\cO)\) defined by
    \begin{align}\label{eq:SymplecticPotentialForExactVolumeForm}
        \eta_{\gamma}(X^u_\gamma)\coloneqq \int_{\gamma}\iota_u\nu,\quad\text{for \(u\in\Gamma(TM)\),}
    \end{align}
    satisfies \(d\eta = \omega^{\rm MW}\).
\end{theorem}
A case of \autoref{thm:SymplecticPotentialForExactVolumeForm} for $m=3$, namely for closed curves in $M$, is studied in the literature  \cite[Proposition 3.5.2]{brylinski2009loop}.

\begin{proof}
    Since each tangent vector of \(\cO\) extends to a fundamental vector field \eqref{eq:XRepresentationOfTO}, it suffices to check \(d\eta(X^u,X^v) = \omega^{\rm MW}(X^u, X^v)\) for fundamental vector fields \(X^u\), \(X^v\).  Apply the exterior derivative formula \(d\eta(X^u,X^v) = X^u (\eta(X^v)) - X^v(\eta(X^u)) - \eta([X^u,X^v])\), \autoref{prop:FundamentalVectorFieldIsLieAlgebraHomomorphism} that \([X^u,X^v] = X^{[u,v]}\), and the derivation formula \eqref{eq:ReynoldsTransportTheorem}:
    \begin{align*}
        d\eta_\gamma(X^u_\gamma, X^v_\gamma) &= \textstyle X^u_\gamma \left(\int_{(\cdot)}\iota_v\nu\right) - X^v_\gamma\left(\int_{(\cdot)}\iota_u\nu\right) - \int_\gamma \iota_{[u,v]}\nu\\
        &=\textstyle\int_\gamma \left(\cL_u\iota_v\nu - \cL_v\iota_u\nu - \iota_{[u,v]}\nu\right)
        \mathrel{\overset{(\star)}=} \int_\gamma\iota_v\iota_u d\nu = \int_\gamma\iota_v\iota_u\mu = \omega^{\rm MW}_\gamma(X^u,X^v).
    \end{align*}
    where the equality \((\star)\) follows from applying the identity \(\cL_u \iota_v\nu = \iota_{[u,v]}\nu + \iota_v\cL_u\nu\) to the first term, and that \(\iota_v\cL_u\nu - \cL_v\iota_u\nu = \iota_v\iota_u d\nu + \iota_v d\iota_u\nu - \iota_v d\iota_u\nu - d\iota_v\iota_u\nu = \iota_v\iota_u d\nu\) by Cartan's formula.  The exact term \(d\iota_v\iota_u\nu\) vanishes under the integral \(\int_\gamma\) using the Stokes theorem and the fact that \(\gamma\) closed.
\end{proof}
Alternatively, \autoref{thm:SymplecticPotentialForExactVolumeForm} can be  proven using a type of calculus for transgressions of differential forms, called the tilda calculus \cite{Vizman2011hatcalculus}.

\autoref{thm:SymplecticPotentialForExactVolumeForm} implies the exactness of the MW form for the Euclidean space \(M = \RR^m\) with the standard volume form \(\mu = dx^1\wedge\cdots\wedge dx^m\). This is because the volume form is exact with potential forms such as
\begin{align}
\label{eq:VolumePotentialOnEuclideanSpace}
    \nu_{\bx}(\bv_1,\ldots,\bv_{m-1}) = {1\over m}\det(\bx,\bv_1,\ldots,\bv_{m-1}).
\end{align}
Applying \autoref{thm:SymplecticPotentialForExactVolumeForm} to \eqref{eq:VolumePotentialOnEuclideanSpace} with \(m=3\) yields the results of previous work \cite{Padilla2019bubbleRings,Tabachnikov2017bicycle} which showed that the MW structure \(\omega^{\rm MW}|_\gamma(\dot\gamma,\mathring\gamma) = \int_{\SS^1}\det(\dot{\paramgam},\mathring\paramgam,\partial_s\paramgam)\, ds\) on the space of closed space curves $\UEmb(\SS^1;\RR^3)$ is exact with explicit formula for the 1-form $\eta$:
\begin{align}\label{eq:liouville_classical_space_curves}
    \eta_\gamma(\dot \gamma)
    = \frac{1}{3} \int_{\SS^1} \det(\paramgam, \dot \paramgam, \partial_s \paramgam) \, ds
\end{align}
where $\paramgam$ is any parametrization of $\gamma$. The proof in \cite{Tabachnikov2017bicycle} relies on integration by parts using an explicit parametrization of $\SS^1$, which does not directly extend to higher dimensions.  
\autoref{thm:SymplecticPotentialForExactVolumeForm} with \eqref{eq:VolumePotentialOnEuclideanSpace} extends this previous result to arbitrary dimensions.

\begin{remark}[Exactness of \(\omega^{\rm MW}\) for closed and compact \(M\)]
We are not aware of the existence or non-existence of a symplectic potential \(\eta\) on a general closed and compact manifold \((M,\mu)\). 
There are a few special cases where exactness has been proved. For example, the space of two opposite oriented distinct points on \(M = \SS^2\), is symplectomorphic to the cotangent bundle of a certain space and therefore admits a symplectic potential \cite{Oakley2013twoPoints}. In the general case, we speculate that exactness of the MW form is unlikely since the volume form \(\mu\) is not exact, or at least that explicit expression such as  \eqref{eq:SymplecticPotentialForExactVolumeForm} is unavailable.
\end{remark}



\section{Implicit Representations of Codimension-2 Submanifolds}\label{sec:implicit_representations}
We have reviewed the explicit shape space \(\on{Gr}_{m-2}(M)\) of codimension-2 submanifolds in the ambient manifold \(M\). In this section, we introduce their \emph{implicit representations} and study the geometry of the space of these implicit representations as a fiber bundle over the exact shape space \(\on{Gr}_{m-2}^{\rm ex}(M)\). 

\subsection{Implicit Representations}\label{sec:implicit_representation_submanifolds}

\begin{definition}
    The space  of implicit representations of codimension-2 submanifolds is the subset \(\cF\subset C^\infty(M;\CC)\) of smooth complex-valued functions \(\psi\) such that the zero set \(\psi^{-1}(0)\) is nonempty, and the differential \(d\psi|_{x}\colon T_xM\to \CC\) is surjective at each \(x\in\psi^{-1}(0)\).
\end{definition}

By the implicit function theorem, the zero set of each \(\psi\in\cF\) is a smooth oriented codimension-2 submanifold.  The orientation of the zero set \(\gamma\) of \(\psi\) is defined by the winding number of the argument of \(\psi\):  for any 2-dimensional oriented disk \(\Sigma\) which intersects \(\gamma\) transversely, the winding number \({1\over 2\pi}\int_{\partial\Sigma}\Im{d\psi\over\psi}\) equals the signed number of intersections between \(\Sigma\) and \(\gamma\).

\begin{definition}
    Let \(\Pi\colon \cF\to\on{Gr}_{m-2}^{\rm ex}(M)\) be the map that assigns each \(\psi\in\cF\) its zero set \(\gamma = \Pi\psi \in \on{Gr}^{\rm ex}_{m-2}(M)\).
\end{definition}
Note that the zero set \(\gamma\) of \(\psi\in\cF\) must be exact:  it must be the boundary of an oriented codimension-1 submanifold.  In fact, \(\gamma\) is the boundary of any regular level set of the argument \(\arg(\psi)\) of \(\psi\).

Conversely, for any $\gamma\in \on{Gr}_{m-2}^{\rm ex}(M)$, one can construct an implicit representation \(\psi\in\cF\) for it.  For example, given \(\gamma = \partial\Sigma\), construct an angle function \(\theta \colon M \setminus \gamma \to \RR/2\pi \ZZ\) as the solution to the Dirichlet problem with a jump boundary condition at \(\Sigma\):
    \begin{align}
        \begin{cases}
            \Delta\theta(x) = 0\text{ for \(x\in M\setminus\overline{\im\Sigma}\)},\\
            \lim_{x\to  \Sigma^+}\theta(x) - \lim_{x\to  \Sigma^-}\theta(x) = 2\pi  
        \end{cases}
    \end{align}
where \(\Sigma^{\pm}\) denote the front and back sides of \(\Sigma\) respectively, and $\Delta$ is the Laplacian with respect to an arbitrary Riemannian metric on $M$.  Let \(\rho\in C^\infty(M;\RR_{\geq 0})\) be a non-negative function that is asymptotically the distance function $\on{dist}_\gamma$ near $\gamma$. Then
     \(\psi\coloneqq \rho e^{\ii\theta}\) belongs to \(\cF\) and \(\Pi\psi  = \gamma\). 

The above construction for the phase part \(e^{\ii\theta}\) of \(\psi\) is identical to the so-called \emph{solid angle field} in the case of \(M = \SS^m\) or \(\RR^m\)  \cite{borodzik2020solid, chernIshida2024prequantization}. \autoref{fig:solid_angle_theta} shows examples of solid angle fields for space curves $\Gr_{1}^{\rm ex}(\RR^3)$.

\begin{figure}[htbp]
\centering
\begin{minipage}[b]{0.33\columnwidth}
\includegraphics[width=\textwidth]{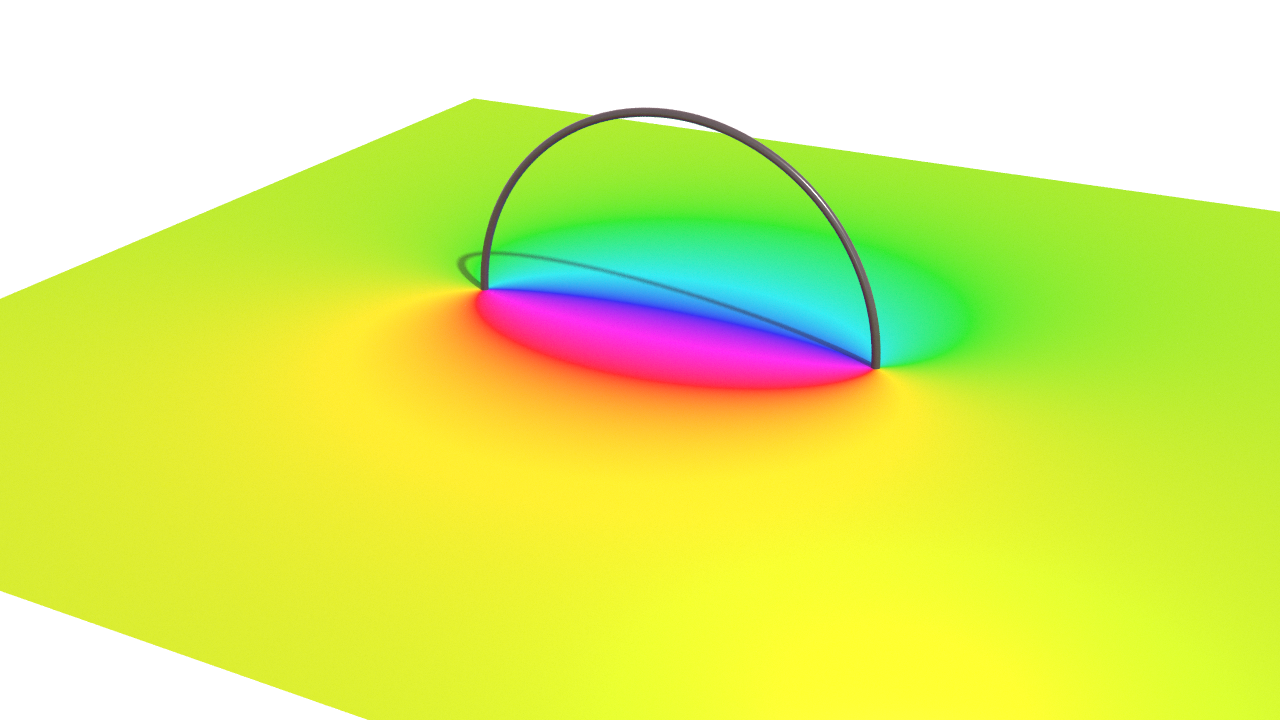}
\end{minipage}
\vrule
\begin{minipage}[b]{0.3\columnwidth}
\includegraphics[width=\textwidth]{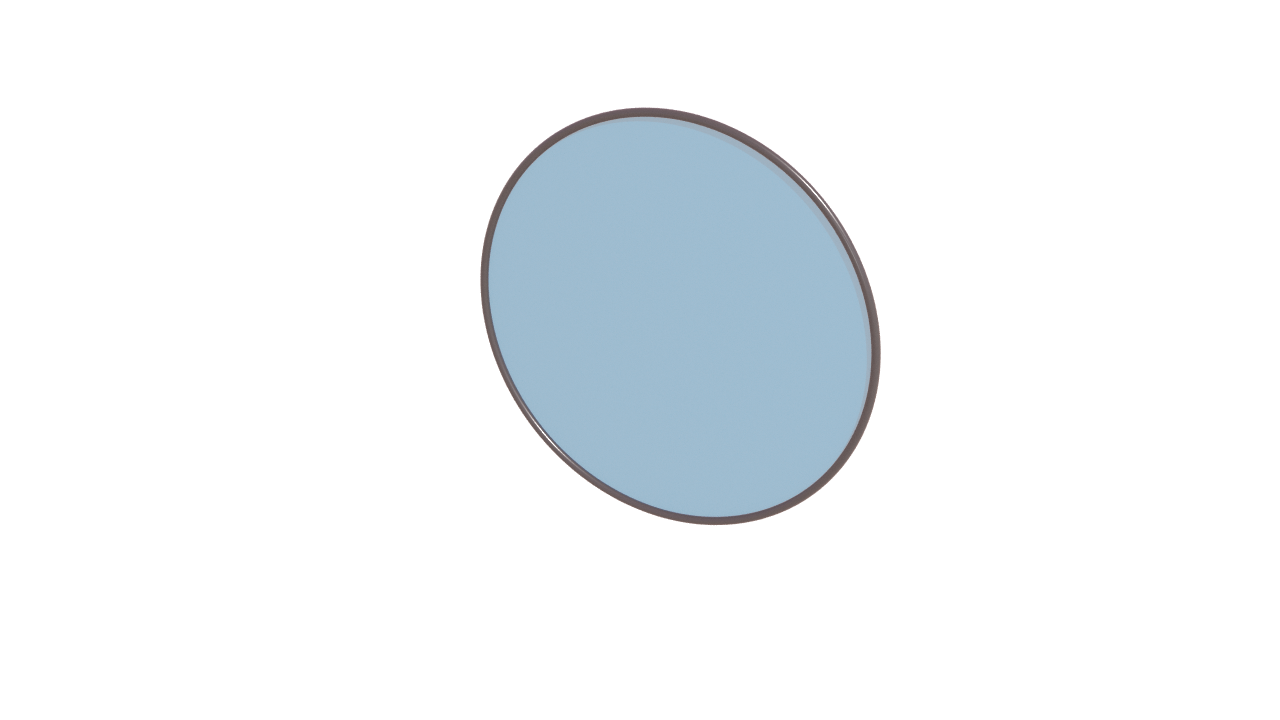}
\end{minipage}
\begin{minipage}[b]{0.3\columnwidth}
\includegraphics[width=\textwidth]{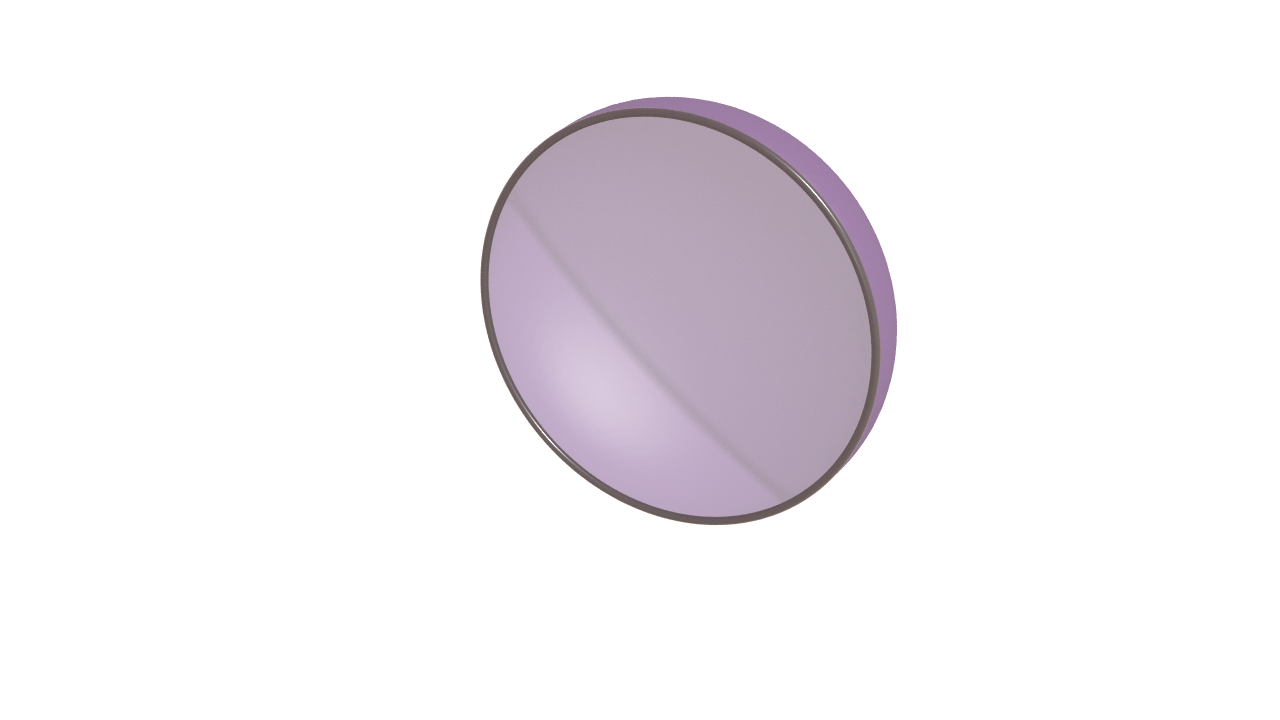}
\end{minipage}
\hrule
\begin{minipage}[b]{0.33\columnwidth}
\includegraphics[width=\textwidth]{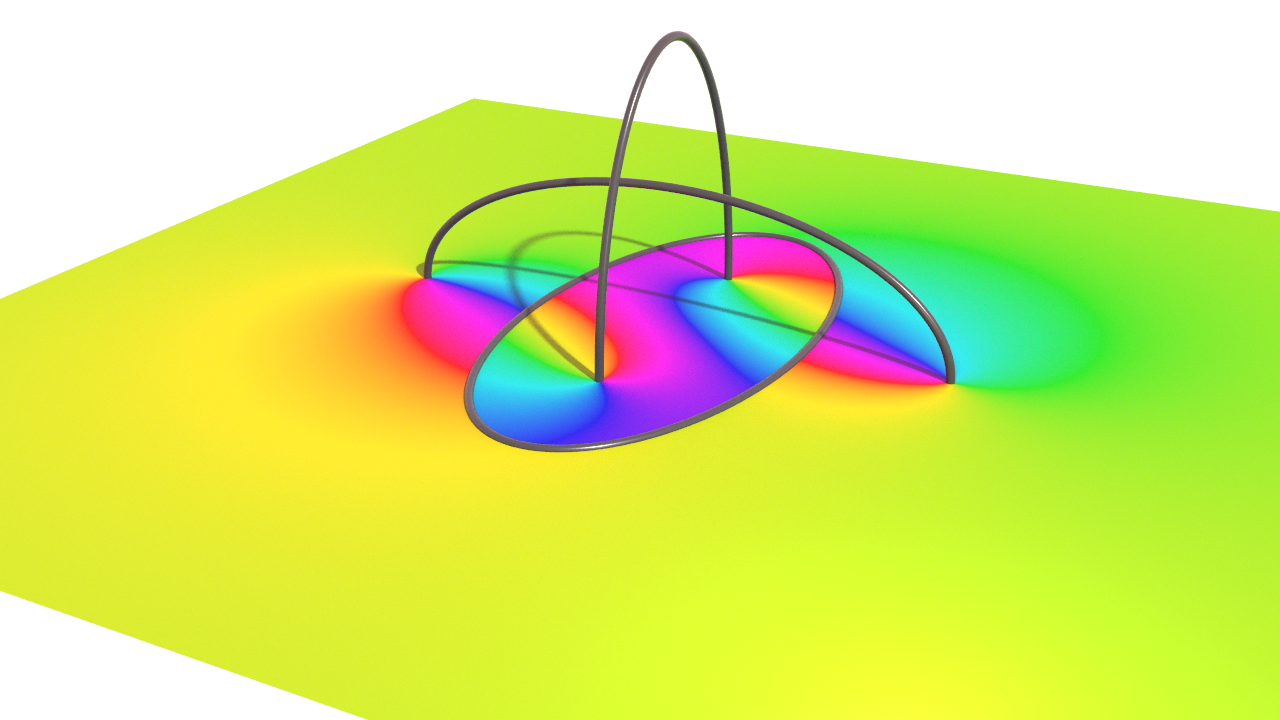}
\end{minipage}
\vrule
\begin{minipage}[b]{0.3\columnwidth}
\includegraphics[width=\textwidth]{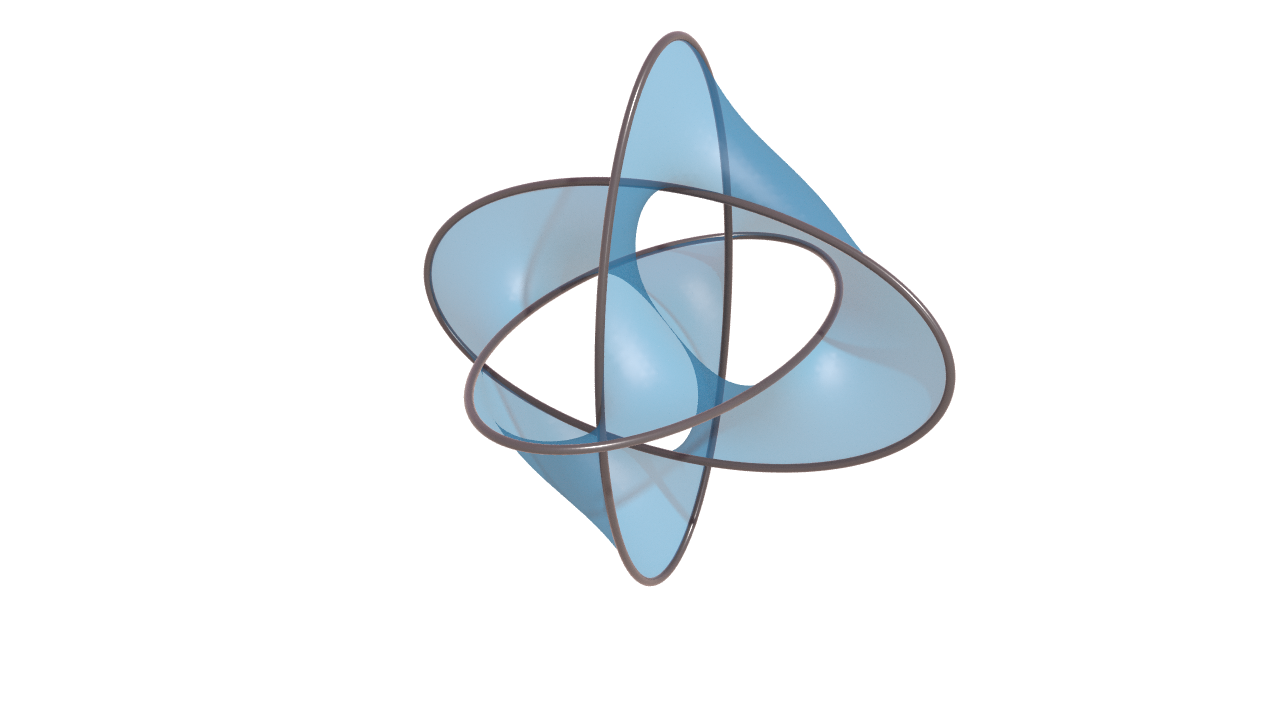}
\end{minipage}
\begin{minipage}[b]{0.3\columnwidth}
\includegraphics[width=\textwidth]{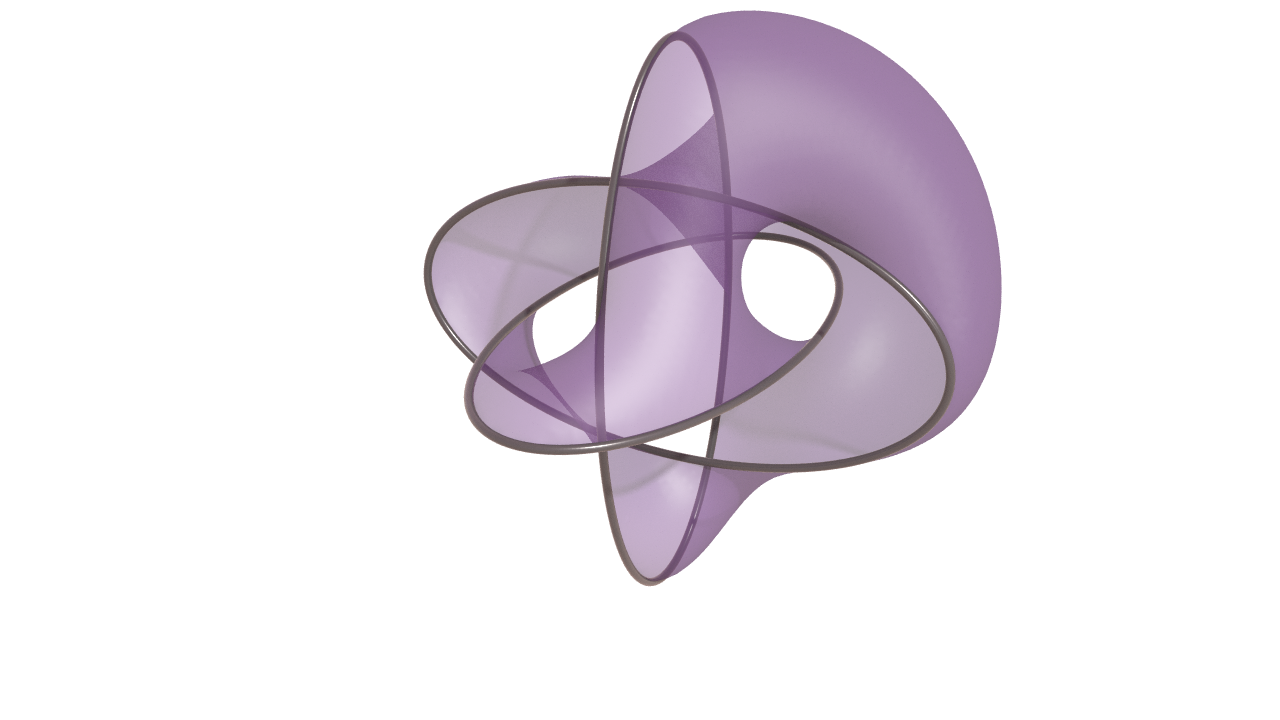}
\end{minipage}
\caption{
Visualization of implicit representations $\psi\in \cF$ for a submanifold $\gamma$ given as a circle (top row) and a link of three circles called Borromean rings (bottom row). In each row, the left image shows the phase level sets on  the plane \(z=0\), where each color represents a phase \ie, the preimage of an element \(s \in \SS^1\) under the phase map \(\phi= \psi/|\psi|\). The middle and the right images are the preimages of two distinct elements in \(\SS^1\) realized as Seifert surfaces bordered by \(\psi^{-1}(0)\). These Seifert surfaces have codimension 1, and their intersection has codimension 2, representing $\gamma$.}
 \label{fig:solid_angle_theta}
\end{figure}

\subsection{Diffeomorphism and Gauge Action}
Recall that each connected component of the explicit shape space $\on{Gr}_{m-2}^{\rm ex}(M)$ is a $\Diff_0(M)$-orbit $\cO$. It turns out that each connected component of the implicit shape space $\cF$ is also an orbit under some group action (\autoref{prop:tangent space of F is fundamental vector fields}), which we explain now.

The diffeomorphism group \(\Diff_0(M)\) acts on implicit representations \(\psi\in\cF\subset C^\infty(M;\CC)\) by pushing forward by the diffeomorphism:
\begin{align}
    \psi\mapsto \psi\circ f^{-1},\quad \psi\in\cF, f\in \Diff_0(M).
\end{align}
This action on \(\cF\) is consistent with the group action \eqref{eq:MovingShapesInM} on \(\on{Gr}_{m-2}^{\rm ex}(M)\): If \(\psi\) is transported by  \(\psi\mapsto\psi\circ f^{-1}\), then its zero set \(\gamma = \Pi\psi \) is transported by \(\gamma\mapsto f\circ\gamma\).  That is, this diffeomorphism action commutes with \(\Pi\).

The following is another group action.  Observe that the zero set \(\Pi\psi \) of \(\psi\) is invariant under a multiplication by a non-vanishing complex-valued function, represented by the exponential \(e^{\varphi}\) of complex-valued function \(\varphi\in C^\infty(M;\CC)\).%
\footnote{On non-simply-connected \(M\), the space \(C^\infty(M;\CC^\times)\) of non-vanishing complex-valued function is a bigger set containing the image of exponential \(\{e^{\varphi}\,|\, \varphi\in C^\infty(M;\CC)\}\).  This will be elaborated in \autoref{sec:geometry_of_fibers}.  In fact, \(\Exp(C^\infty(M;\CC))\) is the connected component of the multiplicative group \(C^\infty(M;\CC^\times)\) containing the identity \(1\).}
This multiplication can be written as the multiplicative group \(\Exp(C^\infty(M;\CC))\coloneqq \{e^{\varphi}\,|\,\varphi\in C^\infty(M;\CC)\}\) acting on \(\cF\) via
\begin{align}
\label{eq:GaugeAction}
    \psi\mapsto e^{\varphi}\cdot\psi,\quad \psi\in\cF, \varphi\in C^\infty(M;\CC).
\end{align}
We call \eqref{eq:GaugeAction} a \emph{gauge transformation} of the implicit representation \(\psi\).

Combining both the diffeomorphism group and the gauge transformation group, consider the following group.
\begin{definition}
    Define the semi-direct product group
    \begin{align}
    \DC \coloneqq \Diff_0(M) \ltimes \Exp(C^\infty(M;\CC)),
    \end{align}
where $\Diff_0(M)$ acts on the multiplicative group $\Exp(C^\infty(M;\CC)) \coloneqq \{ e^\varphi \mid \varphi \in C^\infty(M;\CC) \}$ via
\begin{align}
    f \rhd e^\varphi = e^\varphi \circ f^{-1}, \quad f \in \Diff_0(M), e^\varphi \in \Exp(C^\infty(M;\CC)).
\end{align}
That is, the semi-direct product group structure is explicitly given by
\begin{align}
    (f,e^{\varphi})(g,e^{\chi}) = (f\circ g, e^{\varphi} \cdot  (e^\chi\circ f^{-1})),\quad f,g\in\Diff_0(M), e^{\varphi},e^{\chi}\in \Exp(C^\infty(M;\CC)).
\end{align}
The Lie algebra of the group \(\DC\) is denoted by $\AlgDC = \diff(M) \ltimes C^\infty(M;\CC)$, whose Lie bracket is given by
\begin{align}\label{eq:LieBracekt_of_AlgSemiD}
    [(u, a), (v, b)] = ([u, v], -\cL_u b + \cL_v a), \quad u, v \in \diff(M),\; a, b \in C^\infty(M;\CC),
\end{align}
where $[u, v]$ denotes the Lie bracket of vector fields on $M$.
\end{definition}

The group \(\DC\) acts on \(\cF\) (left group action) via
\begin{align}
    (f, e^\varphi) \rhd \psi =  e^\varphi\cdot(\psi \circ f^{-1}) , \quad (f, e^\varphi) \in \DC, \psi \in \cF.
\end{align}
That is, the function \(\psi\) is pushed forward by the diffeomorphism \(f\) (via the pullback by \(f^{-1}\)) followed by the gauge transformation by \(e^{\varphi}\).


\subsubsection{Fundamental Vector Field}

\begin{definition}\label{def:FundamentalVectorFieldY}
    Let \(Y\colon \AlgDC\to\Gamma(T\cF)\) denote the fundamental vector field associated with the group action of \(\DC\) on \(\cF\), given by
    \begin{align}
        (Y^{(u,\varphi)})_\psi\coloneqq -\cL_u\psi + \varphi\psi,\quad (u,\varphi)\in \AlgDC.
    \end{align}
\end{definition}

For $(u,a),(v,b)\in \AlgDC$, a direct copmutation shows
\begin{align}
   ( Y^{[(u,a),(v,b)]})_\psi= -\cL_{[u,v]} \psi + (\cL_v a - \cL_u b)\psi.
\end{align}
This result and the same argument as in \autoref{prop:FundamentalVectorFieldIsLieAlgebraHomomorphism} show the following relation:
\begin{proposition}\label{prop:Y anti homomorphism}
    The map \(Y\colon \AlgDC \to \Gamma(T\cF)\) is an anti Lie algebra homomorphism with respect to \eqref{eq:LieBracekt_of_AlgSemiD} and the vector field Lie bracket.  That is,
    \begin{align}
        Y^{[(u,a),(v,b)]} = -[Y^{(u,a)},Y^{(v,b)}].
    \end{align}
\end{proposition}

As in the case of explicit representations \eqref{eq:XRepresentationOfTO}, the tangent spaces of \(\cF\) are represented by the fundamental vector fields:
\begin{proposition}\label{prop:tangent space of F is fundamental vector fields}
    Let $\psi\in\cF$. Then 
    \begin{align}
    T_\psi \cF = \{ Y^{(u,\varphi)}_\psi = -\cL_u \psi + \varphi \psi \mid (u, \varphi) \in \AlgDC \}.
\end{align}
As a direct consequence, each connected component of \(\cF\) is a \(\DC\) orbit.
\end{proposition}
\autoref{prop:tangent space of F is fundamental vector fields} asserts that the time evolution of an implicit representation $\psi\in \cF$ is always described as the sum of the transport of $\psi$ along a vector field on $M$ and the pointwise perturbation by the multiplication of a function.  

\begin{proof}
    Since $\cF$ is an open submanifold of the Fr\'echet manifold $C^\infty(M;\CC)$, we have $ T_\psi \cF= T_\psi C^\infty(M;\CC) = C^\infty(M;\CC)$.
    Hence we need to show $C^\infty(M;\CC) =  \{ -\cL_u \psi + \varphi \psi \mid (u, \varphi) \in \AlgDC \}$ for $\psi\in \cF$.
    
   Clearly, $-\cL_u\psi + \varphi \psi\in C^\infty(M;\CC)$ for any $(u,\varphi)\in \AlgDC$. We now show the converse: any $\dot \psi \in   C^\infty(M;\CC) $ is expressed by $\dot \psi=-\cL_u \psi + \varphi \psi$ with some $(u,\varphi)$. Since $d\psi$ is surjective on $\gamma=\psi^{-1}(0)$, it is also surjective in $V$, a small open tubular neighborhood of $\gamma$. Therefore there is some $\tilde u$ such that $\dot \psi= -\iota_{\tilde u}d\psi= -\cL_{\tilde u}\psi$ in $V$. Take a smaller tubular neighborhood $U$ such that $\bar U\subset V$.  Then define $u$ globaly by setting $u\coloneqq \rho \tilde u$ with some smooth cutoff function $\rho$ which takes $1$ inside $U$ and vanishes outside $V$. Define $r:= \dot \psi + \cL_u\psi$ (so $r=0$ on $U$) and $\varphi:=r/\psi$. Note that $\varphi$ is globally defined (smoothly extends to $\psi^{-1}(0)$) as the tubular region $r^{-1}(0)=U$ is a codimension-0 set whose interior contains the codimension-2 set $\psi^{-1}(0)$. 
   Thus we constructed $(u,\varphi)\in \AlgDC$ yielding $\dot\psi=-\cL_u\psi +\varphi\psi$.
\end{proof}


The fundamental vector fields of the $\DC$ action on $\cF$ descend onto the fundamental vector fields of the $\Diff_0(M)$ action on the base space $\on{Gr}_{m-2}^{\rm ex}(M)$:
\begin{proposition}\label{prop:dPi}
    The differential \(d\Pi_\psi\colon T_\psi\cF\to T_{\Pi\psi }\on{Gr}_{m-2}^{\rm ex}\) of the fibration \(\Pi\colon\cF\to\on{Gr}_{m-2}^{\rm ex}\) at \(\psi\in\cF\) is given by
    \begin{align}
        d\Pi_\psi(Y_\psi^{(u,\varphi)}) = X_{\Pi\psi }^u,\quad u\in\Gamma(TM), \varphi\in C^\infty(M;\CC).
    \end{align}
\end{proposition}
\begin{proof}
    This follows from the fact that the action of the diffeomorphism part of \(\DC\) commute with \(\Pi\), and the gauge action part leaves \(\Pi\) invariant.
\end{proof}

\subsection{Geometry of the Fiber Bundle \(\Pi\colon\cF\to\on{Gr}_{m-2}^{\rm ex}(M)\)}\label{sec:geometry_of_fibers}
Here, we investigate the geometry of the bundle \(\Pi\colon\cF\to\on{Gr}_{m-2}^{\rm ex}(M)\) of implicit representations.  
For this study, we may restrict the bundle to \(\cF_\cO\coloneqq \Pi^{-1}\cO\) for a connected component \(\cO\) of \(\on{Gr}_{m-2}^{\rm ex}(M)\).
We show that each fiber has multiple connected components depending on the cohomology of \(M\). This is a key step toward our construction of a prequantum structure in \autoref{sec:prequantum_bundle}. 


\begin{definition}[Fiber-preserving subgroup]
    For each \(\gamma\in\cO\), define \(\Diff_\gamma(M)\subset\Diff_0(M)\) to be the connected component of the stabilizer of \(\gamma\) (under the action of \(\Diff_0(M)\)) containing the identity. That is, \(\Diff_\gamma(M)\) is the subgroup of $\Diff_0(M)$ consisting of diffeomorphisms $f$ for which there exists a path $\{f_t\}_{t \in [0,1]} \subset \Diff_0(M)$ such that $f_t \circ \gamma = \gamma$ for all $t \in [0,1]$, with $f_0 = \id_M$ and $f_1 = f$.
    Define \(\DC_\gamma \coloneqq \Diff_\gamma(M) \ltimes \on{Exp}(C^\infty(M; \CC))\) as the corresponding subgroup of \(\DC\).
\end{definition}

The action on \(\cF\) by the subgroup \(\DC_\gamma\) preserves the fiber \(\Pi^{-1}\gamma\).  This \(\DC_\gamma\)-action on \(\Pi^{-1}\gamma\) is in fact transitive on each connected component of the fiber (\autoref{cor:TransitiveActionOnFiber}).


The following two examples show that both components of the \(\DC_\gamma\)-action,  $\Diff_\gamma(M)$ and $\on{Exp}(C^\infty(M; \CC))$,  must work together to achieve connected componentwise transitivity.

\begin{example}[Non-transitivity of $\Diff_\gamma(M)$-action]\label{eg:non_transitivity_twist}
    Let $\gamma$ be four points $\{z_k\}_{k=1}^4$ on $M=\RR^2$ where  where $z_1, z_2$ have positive orientations and $z_3,z_4$ have negative orientations.
     We then let $\psi_0$ be an implicit representation of $\gamma$ defined by
    \(\psi_0(z)=(z-z_1)(z-z_2)(\overline{z-z_3})(\overline{z-z_4})\),
     and  $\psi_1 \coloneqq \psi_0 e^{\ii\pi/2}$ as illustrated in \autoref{fig:BS_psis}.
    These two functions clearly lie in the same fiber $\Pi^{-1}\gamma$, but
    there is no diffeomorphism $f$ such that $\psi_1 = \psi_0 \circ f$ resolving the topological differences of the phase level sets. In contrast, the group action of $e^{\ii\pi/2} \in \on{Exp}(C^\infty(M; \CC))$ joins $\psi_0$ and $\psi_1$. 
\end{example}

\begin{figure}[htbp]
\centering
\begin{minipage}[b]{0.4\columnwidth}
\includegraphics[width=\textwidth]{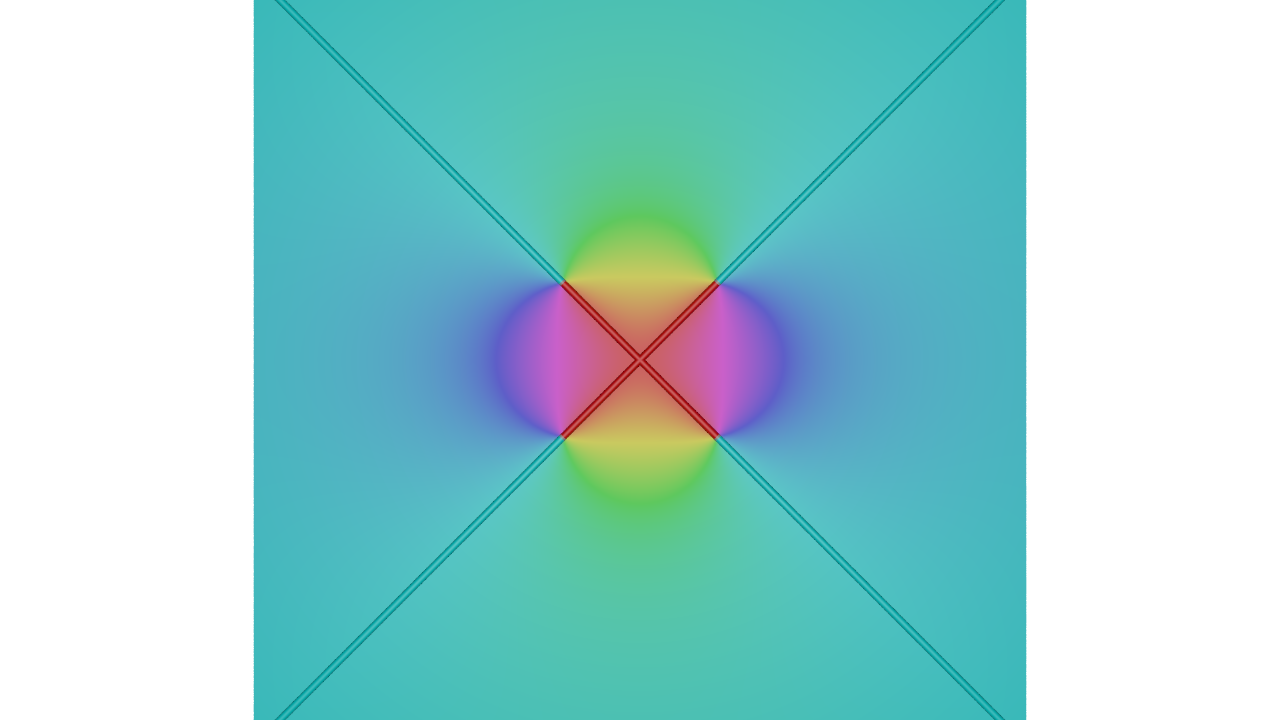}
\par\vspace{-5pt}
\caption*{$\psi_0$}
\end{minipage}
\begin{minipage}[b]{0.4\columnwidth}
\includegraphics[width=\textwidth]{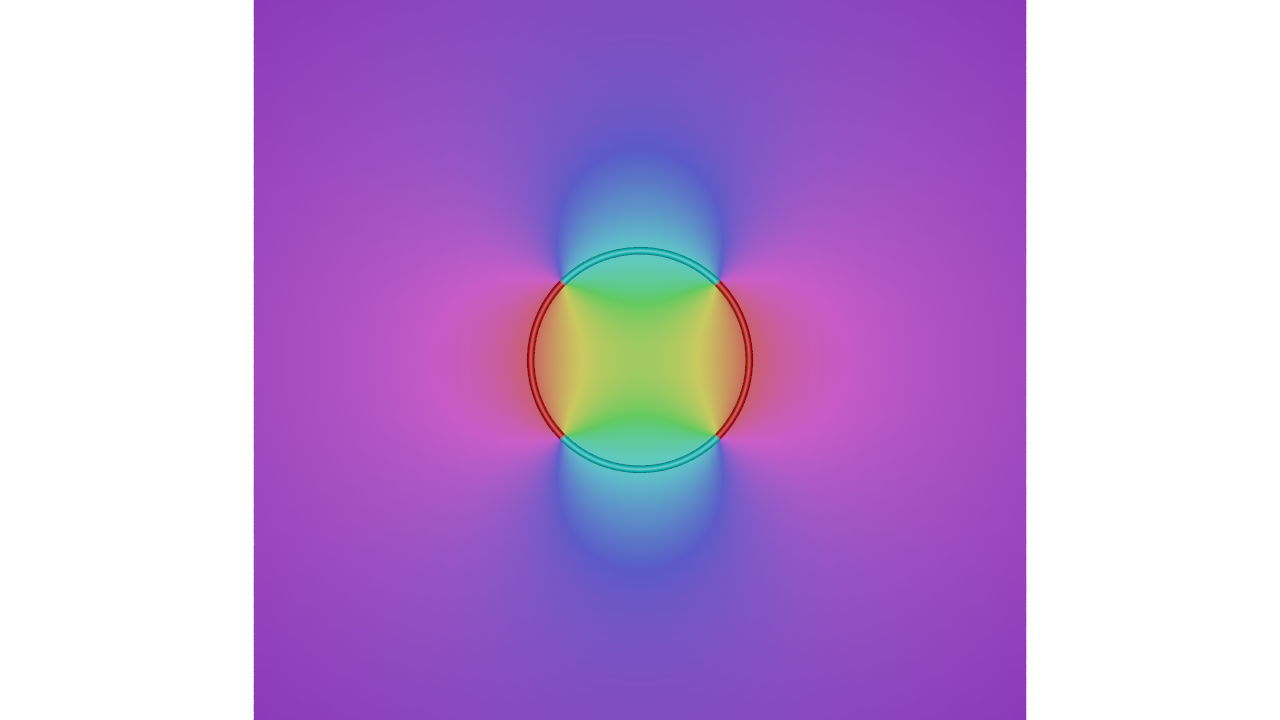}
\par\vspace{-5pt}
\caption*{$\psi_1 = \psi_0 e^{\ii\frac{\pi}{2}}$}
\end{minipage}
\vspace{-5pt}
\caption{Implicit representations of four points in $\RR^2$. Each color indicates a phase value of \(\phi_i = \psi_i / |\psi_i|\). The highlighted light-blue and red curves correspond to the level sets $\phi^{-1}(1)$ and $\phi^{-1}(e^{\ii\pi})$, respectively. Clearly, there is no diffeomorphism $f$ such that $\psi_0 \circ f = \psi_1$ that can handle the topological changes in these level sets.}
\label{fig:BS_psis}
\end{figure}

\begin{example}[Non-transitivity of $\on{Exp}(C^\infty(M; \CC))$-action]\label{eg:non_transitivity_conformal}
Let $M = \RR^2$, and consider implicit representations $\psi_0, \psi_1$ defined by $\psi_0(x, y) = x + \ii y$ and $\psi_1(x, y) = x + y + \ii y$. These functions share the same zero at $(x, y) = (0, 0)$ and the same orientation. However, there is no nowhere-vanishing smooth function $\varphi$ such that $\psi_1 = \varphi \psi_0$. Indeed, the quotient 
\(\frac{\psi_1}{\psi_0} = 1 + \frac{xy}{x^2 + y^2} - \ii \frac{y^2}{x^2 + y^2}\)
is discontinuous at the origin. This is caused by the fact that $\psi_0$ and $\psi_1$ have different  rates of phase change around the origin, as illustrated in \autoref{fig:phase_lines}.  On the other hand, with the diffeomorphism $f(x, y) = (x - y, y)$, we attain $ \psi_0 \circ f=\psi_1$.
\begin{figure}[htbp]
\centering
\begin{minipage}[b]{0.32\columnwidth}
\includegraphics[width=\textwidth]{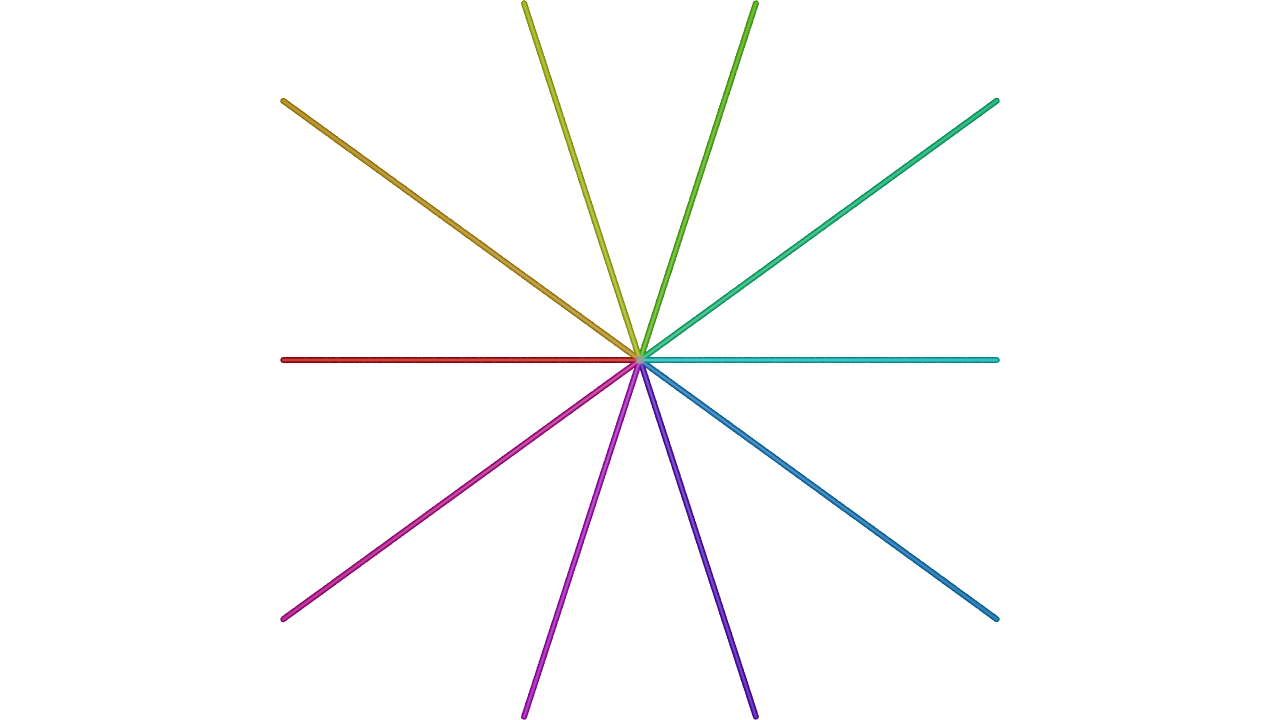}
\captionsetup{margin=-0pt, singlelinecheck=false, justification=centering}
\caption*{$\psi_0(x,y)=x+\ii y$}
\end{minipage}
\hspace{20pt}
\begin{minipage}[b]{0.32\columnwidth}
\includegraphics[width=\textwidth]{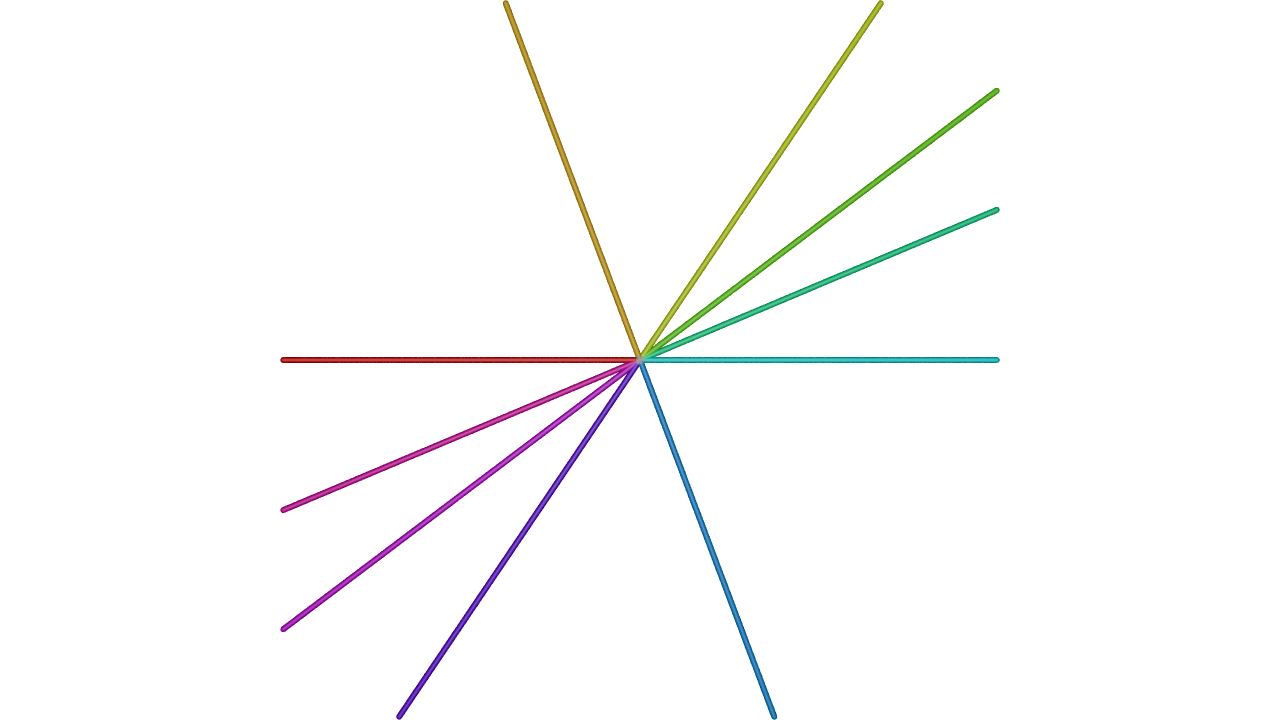}
\captionsetup{margin=-0pt, singlelinecheck=false, justification=centering}
\caption*{$\psi_1(x,y)=x+y+\ii y$}
\end{minipage}
\caption{Level lines of the phase functions $\phi_i=\psi_i/|\psi_i|$. The phase change of  $\psi_1$ around the origin is different from that of $\psi_0$, producing the \emph{sheared} phase level sets.   }
\label{fig:phase_lines}
\end{figure}
\end{example}

Although the ambient manifolds \(M = \RR^2\) for both examples are non-compact, this non-compact assumption is not essential for the construction.
For instance, on \(M = \SS^2\), regarded as the one-point compactification of \(\RR^2\), the same choice of pairs \(\psi_0, \psi_1\) as in \autoref{eg:non_transitivity_twist} and \autoref{eg:non_transitivity_conformal} works, with the modification that \(|\psi_0|\), \(|\psi_1| \to 1\) at infinity so that they are smooth functions on $\SS^2$.

We now characterize the geometry of each fiber \(\Pi^{-1}\gamma\) by studying the behavior of the \(\DC_\gamma\) action on the fiber.

\subsubsection{A Transitive Action for each Fiber \(\Pi^{-1}\gamma\)}
The fiber \(\Pi^{-1}\gamma\subset \cF\) based at \(\gamma\in\cO\) consists of complex-valued functions \(\psi\) with the zero set \(\gamma\).
Given two functions \(\psi_0, \psi_1\) in the same fiber, \ie\@ sharing the same oriented zero set \(\gamma\), it is not \emph{a priori} obvious how they are related to each other.  For example, even when \(\psi_0,\psi_1\) share the same zero set, \(\psi_1/\psi_0\) is generally undefined at their common zero set (\autoref{eg:non_transitivity_conformal}).
We therefore have the following nontrivial classification of elements in \(\Pi^{-1}\gamma\).
\begin{definition}
    Two functions \(\psi_0, \psi_1\in\Pi^{-1}\gamma\) are said to be in the same \emph{conformal class} if \(\psi_1 = \tau\psi_0\) for some \(\tau\in C^\infty(M;\CC^\times)\).
\end{definition}
Intuitively, \(\psi_0\) and \(\psi_1\) are in the same conformal class if they share the same shear states (\ie, the rate of phase change around the zeros) so that \(\psi_1/\psi_0\) smoothly exists on the zeros.

The functions in \(\Pi^{-1}\gamma\) encode not only the zero set \(\gamma\) but also a frame over \(\gamma\): Any regular level set of the phase map \(\phi= \psi/|\psi|\colon M\setminus\gamma\to\SS^1\) yields a distinguished direction field at \(\gamma\) that defines a \emph{ribbon} over the codimension-2 submanifold \(\gamma\) (\autoref{fig:ribbon}).  Intuitively, two ribbons can be flown to each other by a diffeomorphism in \(\Diff_\gamma(M)\)  if an donly if the ribbons share the same total twist. The following \emph{twist class} characterizes this global topological type of the ribbon data within \(\psi\).
\begin{definition}
    Two elements \(\psi_0,\psi_1\in\Pi^{-1}\gamma\) are in the same \emph{twist class} if there exists \(f\in\Diff_\gamma(M)\) such that \(\psi_1 = \psi_0\circ f\).
\end{definition}

\begin{figure}[htbp]
\centering
\includegraphics[width=0.5\textwidth,, trim={4cm 4cm 4cm 4cm},clip]{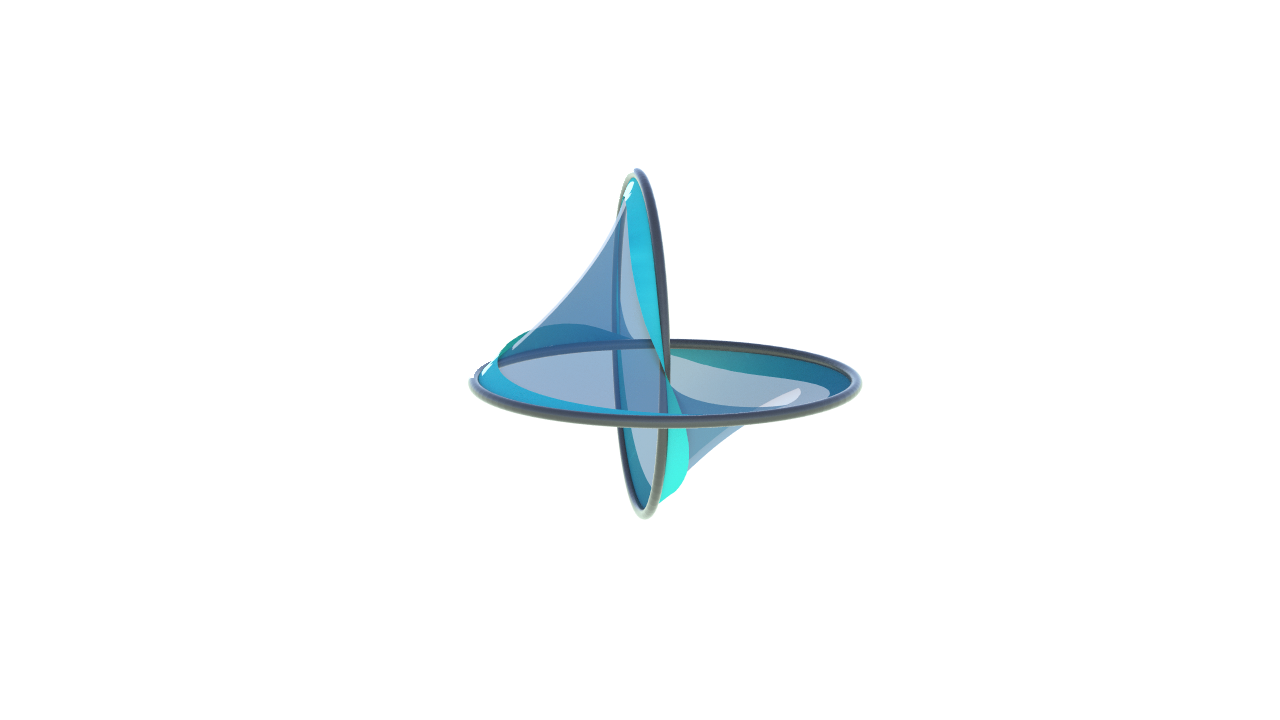}
\caption{A ribbon of an implicit representation $\psi$ for the Hopf link $\gamma$. For each regular value $s\in\SS^1$ of the phase field \( \phi = \psi/|\psi| \), the ribbon \( R_s \) (opaque cyan) is defined as the intersection of \( \phi^{-1}(s) \) (translucent blue) and a small tubular neighborhood of \( \im\gamma \).}
\label{fig:ribbon}
\end{figure}

\begin{proposition}\label{prop:conformal_class_in_fiber}
    Let \(\psi_0,\psi_1\in\Pi^{-1}\gamma\), then there exist \(\tilde\psi_0, \tilde\psi_1\in\Pi^{-1}\gamma\) within the twist classes of $\psi_0$ and $\psi_1$,
    such that \(\tilde\psi_0,\tilde\psi_1\) are in the same conformal class. Equivalently, $\psi_1\circ f=\tau\psi_0$ for some $f\in \Diff_\gamma(M)$ and $\tau \in C^\infty(M;\CC^\times)$.
\end{proposition}
The following is an immediate consequence.
\begin{corollary}\label{cor:TransitiveActionOnFiber}
    The group \(\Diff_\gamma(M)\ltimes C^\infty(M;\CC^\times)\) acts on \(\Pi^{-1}\gamma\) transitively.
\end{corollary}
To prove \autoref{prop:conformal_class_in_fiber},  we explicitly construct $\tilde \psi_1\coloneqq \psi_1\circ f\in\Pi^{-1}\gamma$ with some $f\in \Diff_\gamma(M)$ such that
$ \tilde \psi_1=\tau \psi_0$ for some $\tau \in C^\infty(M; \CC^\times)$ with the following strategy. As we observed, the obstacle for the existence of a smooth quotient function is that $\psi_0$ and $\psi_1$ may have different sheared states (rates of the phase change around the zero sets), which can be remedied using a diffeomorphism on $M$ as in \autoref{eg:non_transitivity_conformal}.
Similarly, to show \autoref{prop:conformal_class_in_fiber}, we will find a diffeomorphism $f\in \Diff_\gamma(M)$ which adjusts the sheared state of $\psi_1$ to that of $\psi_0$, while $f=\id$ away from $\gamma$.

\begin{proof}
Let $\psi_0, \psi_1\in\Pi^{-1}\gamma$.
Take an arbitrary Riemmanian metric on $M$ and a small tubular neighborhood $B$ of $\gamma$. 
Since $d\psi_0$ is surjective on $\gamma=\psi_0^{-1}(0)$, the inverse function theorem ensures that  $\psi_0$ is injective  on each normal disc within $B$.

This allows us to take coordinates  $(s,z)$ on a sub-neighborhood $B_\epsilon:=B\cap \{|\psi|<\epsilon\}$  for a small $\epsilon>0$, where $s\in\gamma$ and $z$ is given by the value of $\psi_0$.
By design $\psi_0^s:=\psi_0(s,\cdot):D_\epsilon \to D_\epsilon$ for each $s\in \gamma$ is the identity map on the disc $D_\epsilon\subset \CC$.  


 Let $\psi_1^s:=\psi_1(s,\cdot):D_\epsilon \to \CC$ for each $s$. Since $\psi_1$ lies in the same fiber as $\psi_0$, we have $\det d\psi_1^s>0$ on $D_\epsilon$, provided that $\epsilon$ is sufficiently small. In particular, $\psi_1^s$ is an orientation-preserving diffeomorphism between $D_\epsilon$ and $\psi_1^s(D_\epsilon)$. 

It follows that there is a family of diffeomorphisms $\{f^s\}_{s\in \gamma}$ on $D_\epsilon$
such that $\psi_1^s\circ f^s (z)= c_s z=c_s \psi_0^s(z)$ with some constant $c_s \in \CC^\times$ on the small disc $D_{\epsilon/3}$, and $f^s=\id$ outside the larger disc $D_{2\epsilon/3}$, as illustrated in \autoref{fig:psi0 to psi1 f}. Here $f_s$ and  $c_s$  depend smoothly on $s\in \gamma$.
 Define $f$ using $\{f^s\}_{s\in \gamma}$ in the tubular neighborhood $B_\epsilon$ and set $f=\id$ outside $B_\epsilon$. Finally, we set $\tilde\psi_1\coloneqq \psi_1\circ f$. Clearly,  $\tilde \psi_1(s,z)/\psi_0(s,z)=c_s$ on $ B_{\epsilon/3}\cong \gamma\times D_{\epsilon/3}$. This shows the existence of a quotient function $\tau\in C^\infty(M;\CC^\times)$ such that $\tilde \psi_1=\tau \psi_0$.

\begin{figure}[htbp]
\centering
\begin{minipage}[b]{0.32\columnwidth}
\includegraphics[width=\textwidth,trim={5cm 0 5cm 0},clip]{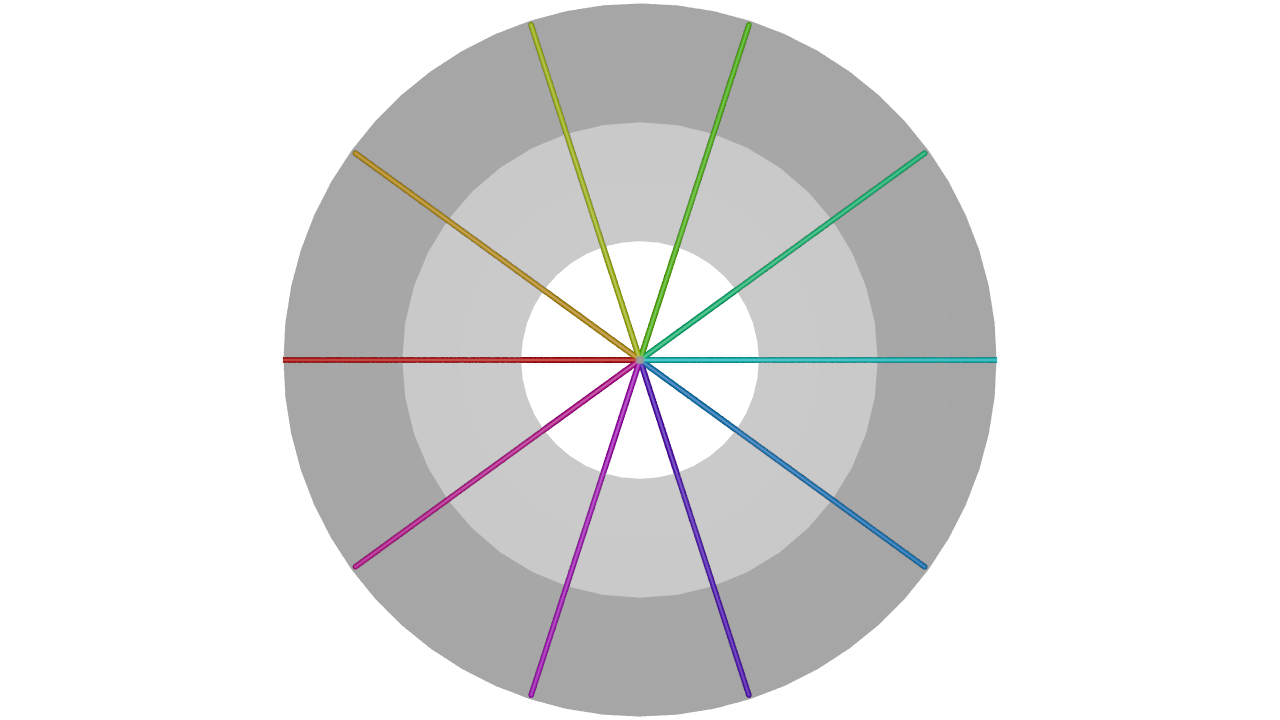}
\captionsetup{margin=-0pt, singlelinecheck=false, justification=centering}
\caption*{ $\psi^s_0(x,y)=x+\ii y$}
\end{minipage}
\begin{minipage}[b]{0.32\columnwidth}
\includegraphics[width=\textwidth,trim={5cm 0 5cm 0},clip]{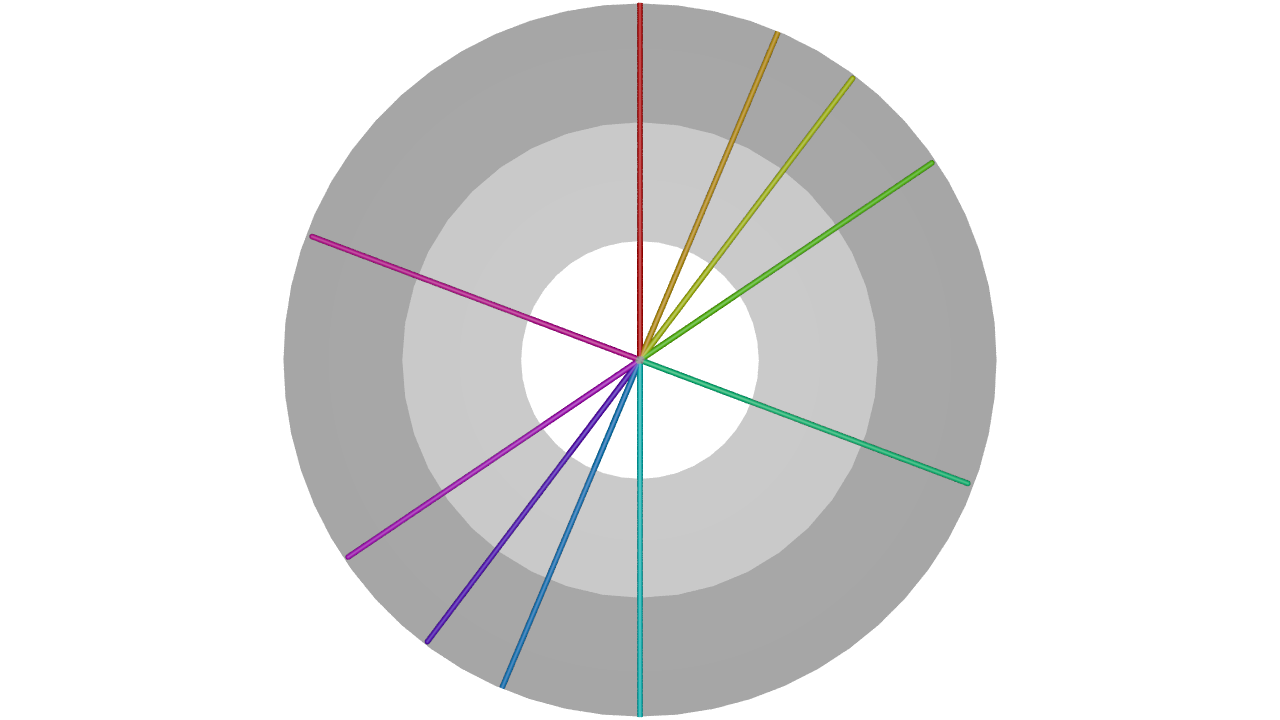}
\captionsetup{margin=-5pt, singlelinecheck=false, justification=centering}
\caption*{$\psi^s_1(x,y)=x+\ii x+\ii y$}
\end{minipage}
\begin{minipage}[b]{0.32\columnwidth}
\includegraphics[width=\textwidth,trim={5cm 0 5cm 0},clip]{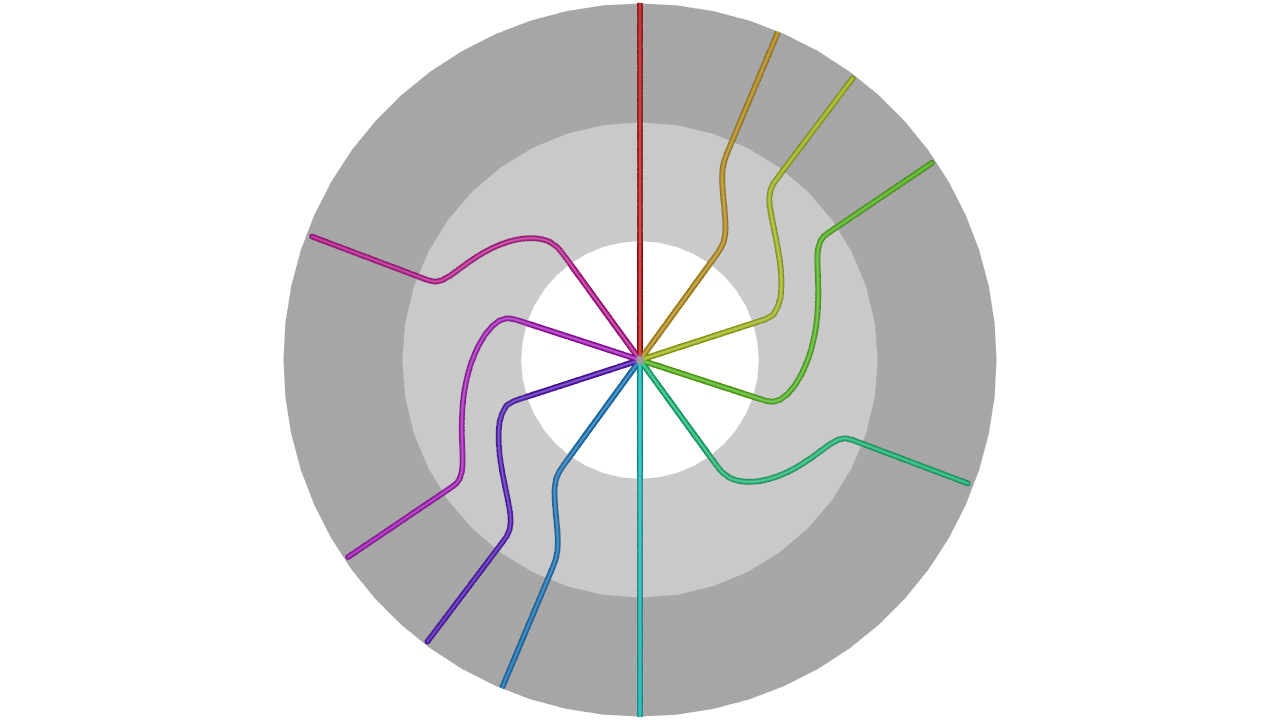}
\captionsetup{margin=40pt, singlelinecheck=false, justification=centering}
\caption*{ $\hat \psi_1^s= \psi^s_1\circ f^s$}
\end{minipage}
\caption{Level lines of the phase maps $\phi=\psi/|\psi|$ on the disc $D_\epsilon$. On the small disc $D_{\epsilon/3}$ colored in white, the sheared states (\ie, the rates of the phase change around the zero), of $\psi_0^s$ (left) and $\hat \psi_1^s\coloneqq \psi^s_1\circ f^s$ (right) coincide, and hence $ \hat \psi_1^s=c_s \psi_0^s$ where $c_s=e^{-\frac{\pi}{2}\ii}$. In the outer region $D_{\epsilon}\setminus D_{2\epsilon/3}$ colored in dark gray, $\hat \psi_1^s$ (right) agrees with $\psi_1^s$ (middle): namely $f^s=\id$ on $D_{\epsilon}\setminus D_{2\epsilon/3}$.}
\label{fig:psi0 to psi1 f}
\end{figure}
\end{proof}
The family nonzero complex numbers $\{c_s\}_{s\in \gamma}$ in the proof may not be taken as a constant independent of $s\in\gamma$ in general. This relates to the cohomology class in $H^1(M;\ZZ)$, corresponding the quotient function $\tau=\psi_1\circ f/\psi_0$. We discuss this now.

\subsubsection{Connected Components of each Fiber \(\Pi^{-1}\gamma\)}\label{sec:ConnectedComponentsOfEachFiber}

Notice the distinction between \(\Diff_\gamma(M)\ltimes C^\infty(M;\CC^\times)\) in \autoref{cor:TransitiveActionOnFiber} and the group \(\DC_\gamma = \Diff_\gamma(M)\ltimes \Exp(C^\infty(M;\CC))\).  The gauge group \(\Exp(C^\infty(M;\CC))\) in \(\DC_\gamma\) is a subgroup of the multiplicative group \(C^\infty(M;\CC^\times)\) of non-vanishing functions.
More precisely, \(\Exp(C^\infty(M;\CC))\) is the connected component of \(C^\infty(M;\CC^\times)\) that contains the identity, and they are related by the following short exact sequence:
\begin{align}
\label{eq:ExpShortExactSequence}
    \xymatrix{
    0\ar[r]
    &
    \Exp(C^\infty(M;\CC))
    \ar@{^{(}->}[r]
    &
    C^\infty(M;\CC^\times)
    \ar[r]^{\frak h}
    &
    H^1(M;\ZZ)
    \ar[r]
    &0
    }
\end{align}
where \(H^1(M;\ZZ)\) is the first integer coefficient cohomology of \(M\), and the map \({\frak h}\) is given by
\begin{align}
\label{eq:frakhDefinition}
    {\frak h}(\tau)(C) = \oint_{C}{1\over 2\pi\ii}{d\tau\over\tau},\quad \tau\in C^\infty(M;\CC^\times),\quad [C]\in H_1(M;\ZZ).
\end{align}
One can check that \eqref{eq:ExpShortExactSequence} is an exact sequence:  For each non-vanishing function \(\tau\in C^\infty(M;\CC^\times)\), one can construct its logarithm \(\varphi\in C^\infty(M;\CC)\) (for \(\tau = e^\varphi\)) locally by \(\varphi(x) = \int_{x_0}^x {d\tau\over\tau}\), \(x\in M\), with some fixed point \(x_0\in M\); this construction works globally, \ie\@ \(\int_{x_0}^x\) is independent of the path, if and only if \eqref{eq:frakhDefinition} vanishes for all loops \(C\).

The short exact sequence \eqref{eq:ExpShortExactSequence}, together with the fact that \(\Exp(C^\infty(M;\CC))\) is the identity component of \(C^\infty(M;\CC^\times)\), implies that the group of connected components is given by
\begin{align}
    C^\infty(M;\CC^\times)/\Exp(C^\infty(M;\CC)) \cong H^1(M;\ZZ).
\end{align}
Similarly, the component group \((\Diff_\gamma(M)\ltimes C^\infty(M;\CC^\times))/\DC_\gamma\) of \(\Diff_\gamma(M)\ltimes C^\infty(M;\CC^\times)\) is also given by \(H^1(M;\ZZ)\), because the diffeomorphism action on \(C^\infty(M;\CC^\times)\) preserves the components.
In particular, \autoref{prop:tangent space of F is fundamental vector fields} implies the following:
\begin{corollary}\label{cor:DCgamma is transitive on fiber}
    Two elements \(\psi_0,\psi_1\in\Pi^{-1}\gamma\) are in the same connected component of \(\Pi^{-1}\gamma\) if and only if \(\psi_1\) is in the \(\DC_\gamma\) orbit of \(\psi_0\).  In particular, \(\DC_\gamma\) acts transitively on each connected component of \(\Pi^{-1}\gamma\).
\end{corollary}

Each connected component of \(\Pi^{-1}\gamma\) can be explicitly quantified as follows. 
Let $\beta$ be the first Betti number of $M$ and  \(C_1,\ldots, C_{\beta}\)  be loops in \(M\) that form a set of generators of \(H_1(M;\ZZ)\) which do not intersect \(\gamma\). The existence of such loops is ensured by a dimension argument: \(\gamma\) is codimension-2 and each \(C_j\) is one-dimensional. 

For each \(\psi\in \Pi^{-1}\gamma\), define \(n_j(\psi) = {1\over 2\pi\ii}\oint_{C_j}{d\psi\over\psi}\in\ZZ\), \(j=1,\ldots,\beta \). Note that the 1-form \({d\psi\over\psi}\) is well-defined along each \(C_j\) thanks to \(C_j\) not intersecting \(\gamma\).  The integer array \(\bn(\psi) = (n_1(\psi),\ldots,n_{\beta }(\psi))\in\ZZ^{\beta }\) is a coordinate for an affine space associated to the module \(H^1(M;\ZZ)\).  The value of \(\bn\) is invariant under the \(\DC_\gamma\) action, and is shifted freely and transitively by the action \(\psi\mapsto\tau\psi\) for \(\tau\in C^\infty(M;\CC^\times)/\Exp(C^\infty(M;\CC))\cong H^1(M;\ZZ)\). Summarizing these arguments, the following holds.
\begin{proposition}\label{prop:H1Z action on fiber}
    For each $\gamma\in\cO$, the discrete group $H^1(M;\ZZ)$ acts on $\Pi^{-1}\gamma/\DC_\gamma$ freely and transitively. 
\end{proposition}

\subsubsection{Transitivity of \(\DC\) Action on the Bundle \(\cF_\cO\)}
In \autoref{sec:ConnectedComponentsOfEachFiber}, we see that when \(M\) is not simply connected, the subgroup \(\DC_\gamma\) does not act transitively on the entire fiber \(\Pi^{-1}\gamma\).  The group \(\DC_\gamma\) fixes the base point \(\gamma\), which obstructs access to all connected components in the fiber.  If we consider instead the action by the full group \(\DC\), which allows the base \(\gamma\) to move by the action, then we can access different components of the fiber. 

\autoref{fig:DC_orbit_U} illustrates this for the case of \(\dim(M)=2\).  By moving around zeros and performing gauge transformation with functions in \(\Exp(C^\infty(M;\CC))\), we can add nontrivial cohomology component to \(\psi\in\cF_\cO\), \ie\@ move \(\psi\) to a different connected component of the fiber.

We expect that a similar argument holds in higher dimensions and the result remains valid:
\begin{conjecture}\label{conj:F_equals_U}
The $\DC$ action is transitive on the fiber bundle  $\cF_\cO\coloneqq \Pi^{-1}\cO$ over a $\Diff_0(M)$-orbit $\cO$ in $\on{Gr}_{m-2}^{\on{ex}}(M)$. Namely, $\cF_\cO$ is a connected space.
\end{conjecture}

\begin{figure}[htbp]
\centering
\begin{minipage}{0.28\columnwidth}
\caption*{$\psi_0$}
\includegraphics[width=\textwidth, trim={0 -1cm 0 0},clip]{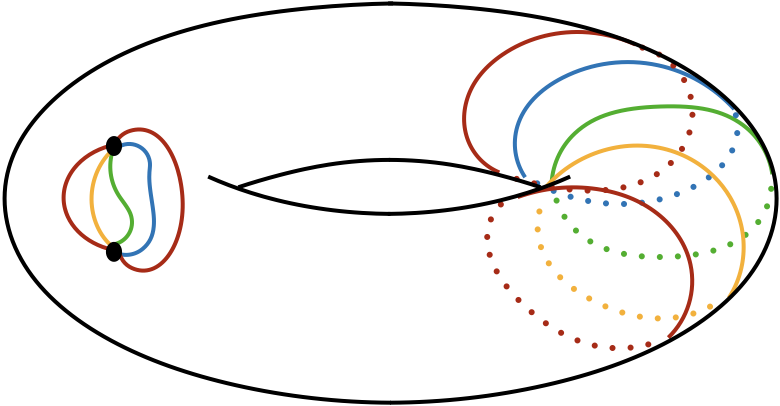}
\end{minipage}
\quad
\begin{minipage}{0.28\columnwidth}
\caption*{$\psi_{t_1}$}
\includegraphics[width=\textwidth, trim={0 -1cm 0 0},clip]{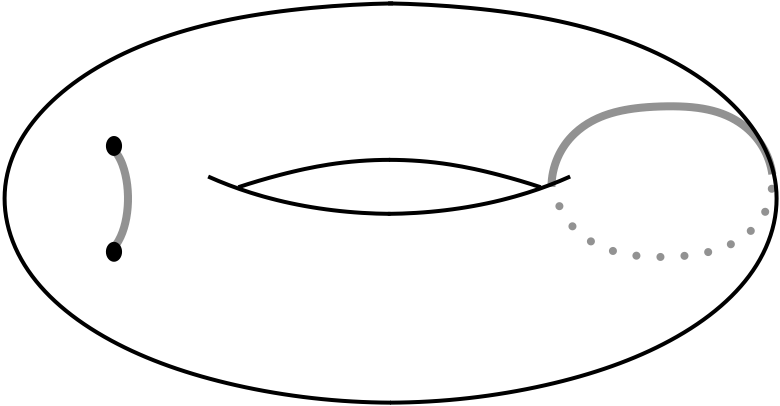}
\end{minipage}
\quad
\begin{minipage}{0.28\columnwidth}
\caption*{$\psi_{t_2}$}
\includegraphics[width=\textwidth, trim={0 -1cm 0 0},clip]{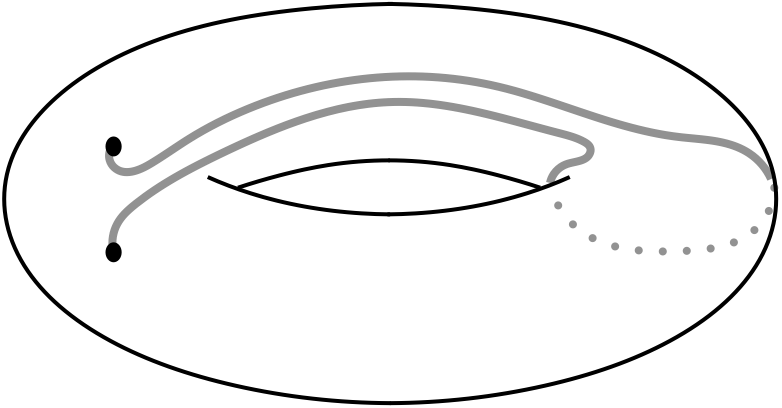}
\end{minipage}
\begin{minipage}{0.28\columnwidth}
\caption*{$\psi_{t_3}$}
\includegraphics[width=\textwidth, trim={0 0 0 0},clip]{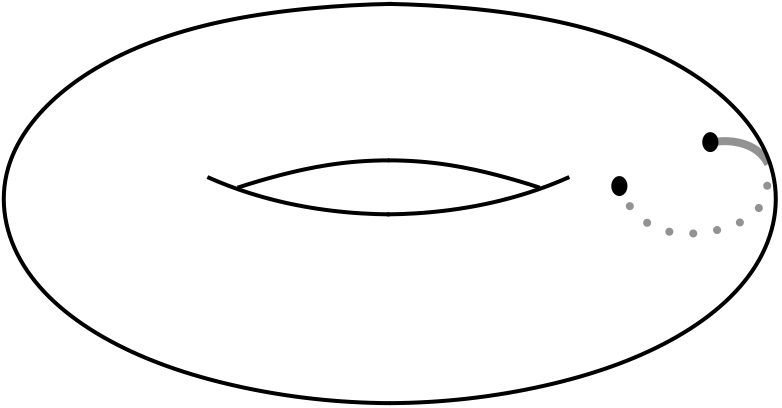}
\end{minipage}
\quad
\begin{minipage}{0.28\columnwidth}
\caption*{$\psi_{t_4}$}
\includegraphics[width=\textwidth, trim={0 0 0 0}, clip]{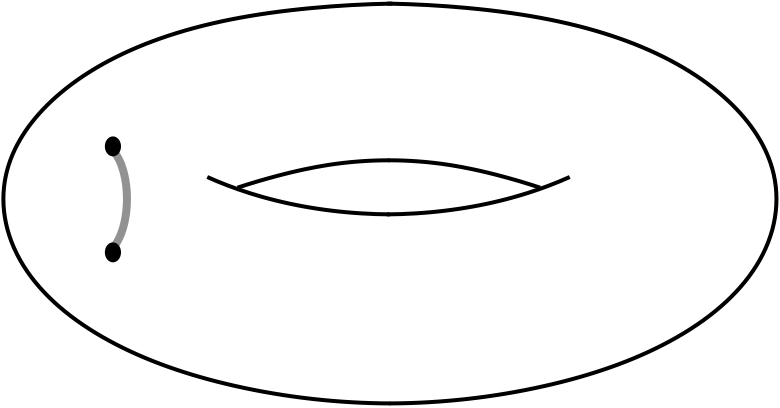}
\end{minipage}
\quad
\begin{minipage}{0.28\columnwidth}
\caption*{$\psi_1$}
\includegraphics[width=\textwidth, trim={0 0 0 0},clip]{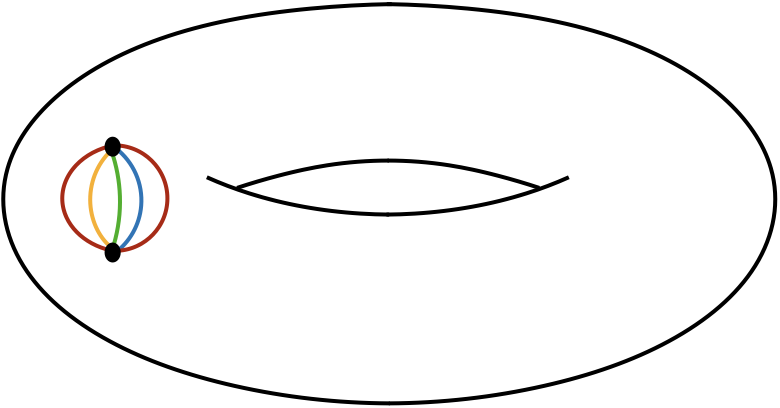}
\end{minipage}
  \caption{A transition of the implicit representation from $\psi_0$ (top left) to $\psi_1$ (bottom right) for two points $\gamma$ on the torus $\TT^2$. In both $\psi_0$ and $\psi_1$, the colored curves represent some level sets of the phase map $\phi\coloneqq \frac{\psi}{|\psi|}:M\setminus\gamma\to \SS^1$. Clearly, $\psi_0$ and $\psi_1$ belong to different connected components of the fiber $\pi^{-1}\gamma$. 
Nevertheless, they lie on the same $\DC$-orbit. Starting from $\psi_0$, we first multiply by some $e^{\varphi_1} \in \Exp(\CC^\infty(M;\CC))$ to \emph{compress} the support of $d\phi$ into narrow bands, visualized as thick gray curves in $\psi_{t_1}$. 
Next, multiplying $\psi_{t_1}$ by $e^{\varphi_2} \in \Exp(\CC^\infty(M;\CC))$ merges two level curves, producing $\psi_{t_2}$. Applying a diffeomorphism $f_1$ moves the level lines, yielding $\psi_{t_3}$, and a second diffeomorphism $f_2$ transforms $\psi_{t_3}$ into $\psi_{t_4}$. Finally, multiplying by another $e^{\varphi_3} \in \Exp(\CC^\infty(M;\CC))$ decompresses the support of $d\phi$, resulting in $\psi_1$.}

\label{fig:DC_orbit_U}
\end{figure}

\section{Prequantum potential}\label{sec:prequantum_potential}

In this section, we present one of our main results: the MW symplectic form $\omega_{\rm{MW}}$ on $\cO$ over a general closed manifold $M$ admits a symplectic potential in the \emph{prequantum sense}. 

\begin{definition} 
Define the following 1-form on  \(\cF\):
\begin{align}\label{eq:Th_integral} \Th_\psi(\dot \psi) = {1\over 2\pi|M|} \int_{M} \Im  \frac{\dot \psi \bar\psi}{|\psi|^2} \  \mu, \quad \dot \psi \in T_\psi \cF.
\end{align}
\end{definition}

The integral in \eqref{eq:Th_integral} may appear divergent as the integrand is unbounded near the zeros of \(\psi\), but we will show that it is finite and has a geometric interpretation. 

\begin{definition}[Formal prequantization and prequantum potential]\label{def:formal_prequantization}
    Let $(B, \beta)$ be a symplectic manifold, and let $E$ be a manifold equipped with a 1-form $\alpha$. We say that a fibration $\pi \colon E \to B$ is a \emph{formal prequantization}, and that $\alpha$ is a \emph{formal prequantum potential} of $\beta$, if
\begin{align}\label{eq:prequantum_condition}
        d\alpha = \pi^* \beta.
    \end{align}
\end{definition}

\begin{theorem}[Formal prequantization of the MW structure] \label{th:quantum_potetial}
    Suppose $M$ is a closed manifold equipped with a volume form $\mu$.
    Then the fibration $\Pi \colon (\cF_\cO, \Th) \to (\cO, \omega^{\rm MW})$ is a formal prequantization.
\end{theorem}
To prove the theorem, that is, to show \(d\Th = \Pi^*\omega_{\rm{MW}}\), compute
\begin{align}\label{eq:exterior_derivative_in_U}
d\Theta_\psi(Y_\psi^{(u,a)},Y_\psi^{(v,b)})
= \cL_{Y^{(u,a)}} \iota_{Y^{(v,b)}} \Th_\psi - \cL_{Y^{(v,b)}} \iota_{Y^{(u,a)}} \Th_\psi - \iota_{[Y^{(u,a)}, Y^{(v,b)}]} \Th_\psi
\end{align}
using the representation in terms of fundamental vector fields $Y^{(u,a)}, Y^{(v,b)}\in \Gamma(T\cF)$ (\autoref{def:FundamentalVectorFieldY}).

To explicitly evaluate each term of \eqref{eq:exterior_derivative_in_U}, it is convenient to express \eqref{eq:Th_integral} in terms of the differential of the \(\SS^1\)-valued function \(\arg(\psi)\).  This differential of the phase is defined as follows.

\begin{definition}[Phase 1-form]
\label{def:circle_differential}
Using the phase map $\phi:M\setminus\psi^{-1}(0)\to \SS^1$ given by $\phi=\psi/|\psi|$  
and the standard Haar measure \(\sigma\in\Omega^1(\SS^1)\) on \(\SS^1\) with normalization \(\int_{\SS^1}\sigma = 2\pi\), define \emph{the phase 1-form \(\lambda_{\psi}\in \Omega^1(M\setminus\psi^{-1}(0))\)} by
\begin{align}
    \lambda_\psi \coloneqq \phi^*\sigma.
\end{align}
This induces the map $\Lambda\colon \cF\to \Omega^{m-1}(M)^*$ by
\begin{align}\label{eq:integral_lambda_psi}
    \langle \Lambda(\psi), \beta \rangle\coloneqq
    \int_{M} \lambda_\psi \wedge \beta, 
    \quad \beta \in \Om^{m-1}(M).
\end{align}
We call $\Lambda(\psi)$ \emph{the phase current}. 
\end{definition}
The phase 1-form \(\lambda_\psi\) represents the gradient of the phase \(\arg(\psi)\) of \(\psi\): On each simply-connected subset \(U\subset M\setminus\psi^{-1}(0)\), express \(\psi\) as \(\psi(x) = r(x)\cdot e^{\ii\vartheta(x)}\) with some \(r\colon U\to\RR_{>0}\) and single-valued function \(\vartheta\colon U\to\RR\); then \(d\vartheta = \lambda_{\psi}\).




\begin{proposition}\label{prop:circle_differential}

For $\psi\in \cF$ and $\beta\in \Omega^{m-1}(M)$, it holds
    \begin{align}\label{eq:circle_phase_integral}
        \langle \Lambda(\psi), \beta \rangle =\int_{\SS^1} \left(\int_{\phi^{-1}(\cdot) } \beta \right)\sigma
    \end{align}
    where $\phi=\psi/|\psi|:M\setminus\psi^{-1}(0)\to \SS^1$ is the phase map of $\psi$.
\end{proposition}
The integral \eqref{eq:circle_phase_integral} makes sense as Sard's theorem asserts that the integrand  is defined for $\sigma$-almost every element of $\SS^1$.  
\begin{proof}
    The stated expression is easily verified using the coarea formula \cite[Chapter 1.3]{demailly1997complex}.
\end{proof}

We now  express the term $\Th_\psi(Y^{(u,a)}) = \Th_\psi(-\cL_u \psi + a\psi)$ more explicitly.

\begin{lemma}[Evaluation of $\Th_\psi$]\label{lem:value_of_Theta}
For any $(u,a) \in \AlgDC$, we have
\begin{align}
   \Th_\psi(-\cL_u \psi + a\psi)
   &= \Th_\psi(-\cL_u \psi) +  \Th_\psi(a \psi) \\
   &= {1\over 2\pi|M|} \left( - \langle \Lambda(\psi), \iota_u \mu \rangle + \int_M \Im a \ \mu \right),\label{eq:evaluation_of_Th}
\end{align}
where $\Lambda(\psi)$, defined in \eqref{eq:integral_lambda_psi}, is the current associated to the circle differential of \(\psi\), and \(\mu\) is the volume form on \(M\).
\end{lemma}

\begin{proof}
A direct calculation yields
$\Th_{\psi}(a\psi)
        = {1\over 2\pi |M|} \int_M \Im \frac{a\psi \bar\psi}{ |\psi|^2}\mu
        = {1\over 2\pi |M|}  \int_M \Im a\ \mu.$
%
We then compute  $\Th_{\psi}(-\cL_u \psi)$. From \autoref{def:circle_differential}, it is straightforward to compute
\begin{align}
    \Im \frac{d\psi}{\psi}=\lambda_\psi\quad\text{on \(M\setminus \psi^{-1}(0)\)}.
\end{align}
Therefore, 
\begin{align}
     \Th_{\psi}(-\cL_u \psi)
&=- {1\over 2\pi|M|} \int_M \Im \frac{\iota_u d \psi  \cdot \bar \psi} {|\psi^2|}\mu
= - {1\over 2\pi|M|} \int_M \iota_u \lambda_\psi \ \mu
\\& = - {1\over 2\pi|M|} \int_M \lambda_\psi \wedge \iota_u \mu
=- {1\over 2\pi|M|}  \langle\Lambda(\psi), \iota_u \mu\rangle.
\end{align}
\end{proof}

\begin{remark}[Geometric interpretation of $\Th_\psi$]\label{rmk:geometric_meaning_of_Th}
\autoref{prop:circle_differential} and \autoref{lem:value_of_Theta} show that $\Th(\dot \psi)$ is finite for any $\dot \psi\in T_\psi \cF$. 
This also tells a geometric insight. Recall that each regular value $s\in \SS^1$ of the phase map \(\phi=\psi/|\psi|\) defines a hypersurface \(\phi^{-1}(s)\) in $M$.
The formula $\Th(-\cL_u \psi)={1\over 2\pi |M|} \int_{\SS^1} \left(\int_{\phi^{-1}(\cdot)}\iota_u\mu \right)\sigma $ describes that \(\Theta(-\cL_u\psi)\) measures the flux \(\int_{\phi^{-1}(s)}\iota_u\mu\) of the vector field \(u\) through the phase level hypersurface \(\phi^{-1}(s)\), averaged over \(s\in\SS^1\).  

On the other hand, the formula \(\Theta(a\psi) = {1\over 2\pi |M|}\int_M\Im a\, \mu\) captures the infinitesimal phase shift of \(\psi\) over the space.


Together, \(\Th_\psi(\dot \psi) = \Theta_\psi(-\cL_u\psi + a\psi)\) combines both contributions of the infinitesimal phase shift and the flux of vector fields through phase hypersurfaces. In other words, it is the average of the infinitesimal swept volume by the $\SS^1$-family of hypersurfaces $\{\phi^{-1}(s)\}_{s\in\SS^1}$. 
In particular, the integral \(\int_{\Psi}\Theta\) of the 1-form \(\Theta\) along any path \(\Psi\colon [0,T]\to \cF\), \(\Psi(t)\eqqcolon\psi_t\), is the signed volume swept out by the hypersurfaces \(\{\phi_t^{-1}(s)\}_{s\in\SS^1}\) over \(t\in[0,T]\), averaged over \(s\in\SS^1\).


\end{remark}

Next, we need an auxiliary result describing how $\Th$ varies under the $\DC$-action. 

\begin{lemma}\label{lem:theta_under_S_action}
Let $\psi\in\cF$, $(f,e^\varphi)\in\DC$, and $\hat \psi\coloneqq (f,e^\varphi)\rhd \psi=\psi\circ f^{-1}\cdot e ^\varphi$, For  the fundamental vector field of $Y^{(u,a)}\in \Gamma(T\cF)$ of $(u,a)\in \AlgDC$, we have 
\begin{align}
    \Theta_{\hat \psi}(Y^{(u,a)}_{\hat\psi})
    = {1\over 2\pi|M|} \left( -\int_M \iota_u ({f^{-1}}^* \lambda_\psi + \Im d\varphi) \mu + \int_M a\,\mu \right).
\end{align}

    
\end{lemma}

\begin{proof}
It is straightforward to verify $ \Th_{\hat\psi}(a \hat \psi)=\int \Im a \ \mu$.
We now compute  $ \Th_{\hat\psi}(-\cL_u \hat \psi)$.
From 
    \begin{align}
        d\hat \psi
        &= d(\psi\circ {f^{-1}} \cdot e^\varphi)
        =e^\varphi \cdot d(\psi\circ {f^{-1}} ) + \psi\circ {f^{-1}} \cdot  e^\varphi d\varphi
        =e^\varphi d(\psi\circ {f^{-1}} ) + \hat \psi d\varphi
    \end{align}
    it follows that, 
    \begin{align}
 \Th_{\hat \psi}(-\cL_u\hat\psi)
&= -{1\over 2\pi|M|} \int_M \Im \frac{\iota_u d\hat\psi \cdot \bar{\hat \psi}}{ |\hat \psi|^2}\mu
        =-{1\over 2\pi|M|}\int_M \Im \frac{e^\varphi \iota_u d(\psi\circ f^{-1})  + \bar{\hat \psi} \hat \psi  \iota_u d\varphi}{ |\hat \psi|^2}\mu\\
        &=-{1\over 2\pi|M|}\int_M \Im \frac{\bar{\hat \psi} \hat \psi i \iota_u  \lambda_{\psi\circ f}+|\hat \psi|^2 \iota_u d\varphi}{|\hat \psi|^2}\mu
        =-{1\over 2\pi|M|}\int_M \left(\iota_u \lambda_{\psi\circ f}+\iota_u \Im d \varphi \right) \mu.
    \end{align}
    Following \autoref{def:circle_differential}, one obtains
     \(\lambda_{\psi\circ f^{-1}}= (\phi\circ f^{-1}) ^*\sigma=f^{-1\,*}\phi^*\sigma=f^{-1\,*}\lambda_{\psi}\) where $\phi=\psi/|\psi|$.
     This yields the stated expression.
    
    
\end{proof}

Using these results, we can now explicitly evaluate each term in \eqref{eq:exterior_derivative_in_U} and prove \autoref{th:quantum_potetial}.

\begin{proof}[Proof of \autoref{th:quantum_potetial}]
    First, we have 
    \begin{align}
       \Theta_\psi([Y^{(u,a)},Y^{(v,b)}])
       ={1 \over 2\pi |M|}\int_M  \left( -\iota_{[u,v]}\lambda_\psi +\cL_u \Im b -\cL_v \Im a\right) \mu
    \end{align}
   using the expressions of $[Y^{(u,a)},Y^{(v,b)}]$  (\autoref{prop:Y anti homomorphism}) and of $\Theta_\psi$ (\autoref{lem:value_of_Theta}). 
    
    Next, we evaluate $\cL_{Y^{(u,a)}}\iota_{Y^{(v,b)}}\Theta_\psi$. 
    Let \(\Phi_{Y^{(u,a)}}^t\colon \cF\to\cF\) be the flow map generated by the vector field \(Y^{(u,a)}\), explicitly given by \(\Phi^t_{Y^{(u,a)}}(\psi)={\on{Fl}_u^{-t}}^*\psi\cdot e^{ta}\). Here \(\on{Fl}_u^t\colon M\to M\) is the flow map  generated by \(u\in \Gamma(TM)\), given as \({d\over dt}\on{Fl}_u^t = u\circ\on{Fl}_u^t, \on{Fl}_u^0 = \id_M\).  Applying \autoref{lem:theta_under_S_action} to the \(\DC\)-action \(\Phi_{Y^{(u,a)}}^t\) yields:
    \begin{align}
        \cL_{Y^{(u,a)}}\Theta(Y^{(v,b)}(\psi))
        &=\frac{d}{dt}\Big|_{t=0}\Theta(Y^{(v,b)}(\Phi^t_{Y^{(u,a)}}(\psi)))\\
        &={1 \over 2\pi |M|}\left(\frac{d}{dt}\Big|_{t=0}\int -\left(\iota_v {\on{Fl}_{u}^{-t}}^*\lambda_\psi + \iota_v d\Im (t a)\right) \mu + \int \Im b\, \mu \right) \\
        &={1 \over 2\pi |M|}\int \left(\iota_v \cL_u\lambda_\psi - \cL_v \Im a\right) \mu.
    \end{align}
    Similarly, \(\cL_{Y^{(v,b)}}\Theta(Y^{(u,a)}(\psi))={1 \over 2\pi |M|}\int \left(\iota_u \cL_v \lambda_\psi -\cL_u \Im b  \right)  \mu \).

    These expressions allow us to expand \eqref{eq:exterior_derivative_in_U}:
    \begin{align}
       d\Theta_\psi(Y^{u,a}, Y^{v,b})
       &=  \cL_{Y^{(u,a)}}\iota_{Y^{(v,b)}} \Theta_\psi - \cL_{Y^{(v,b)}} \iota_{Y^{(u,a)}}\Theta_\psi - \iota_{[Y^{(u,a)},Y^{(v,b)}]}\Theta_\psi\\
       &={1 \over 2\pi |M|}\int \Big(\iota_v \cL_u\lambda_\psi - \cL_v \Im a - \iota_u \cL_v \lambda_\psi + \cL_u \Im b \\
       &\qquad\quad\quad + \iota_{[u,v]}\lambda_\psi + \cL_v \Im a - \cL_u \Im b \Big) \mu\\
       &={1 \over 2\pi |M|}\int \left( \iota_v \cL_u\lambda_\psi - \iota_u \cL_v \lambda_\psi + \iota_{[u,v]}\lambda_\psi\right) \mu.
    \end{align}
    Using the formula $\iota_{[u,v]}=\cL_u\iota_v - \iota_v\cL_u $, the integrand can be further simplified as \(\iota_v \cL_u \lambda_\psi - \iota_u \cL_v \lambda_\psi +\iota_{[u,v]}\lambda_\psi =  \iota_v\iota_u d \lambda_\psi - d\iota_v \iota_u\lambda_\psi = \iota_v\iota_u d \lambda_\psi -0 \).
Hence we obtain by the product rule that, 
\begin{align}
     \iota_{Y^{(v,b)}} \iota_{Y^{(u,a)}}d\Theta_\psi 
    &= {1 \over 2\pi |M|}\int \iota_v\iota_u d \lambda_\psi \ \mu
    = -{1 \over 2\pi |M|}\int d \lambda_\psi \wedge \iota_u\iota_v\mu
    = - {1 \over 2\pi |M|}\int \lambda_\psi \wedge d\iota_u\iota_v\mu.  
\end{align}
Finally, we apply \autoref{prop:circle_differential} to get
\begin{align}
    \int \lambda_\psi \wedge d\iota_u\iota_v\mu 
     =  \int _{\SS^1}\left( \int_{\phi^{-1}(\cdot)} d\iota_u\iota_v\mu \right)\sigma
\end{align}
where the integrand is defined for almost every $s\in \SS^1$ due to Sard's theorem. By the Stokes theorem, we have
\begin{align}
     -{1 \over 2\pi |M|}\int _{\SS^1}\left( \int_{\phi^{-1}(\cdot)} d\iota_u\iota_v\mu \right)\sigma
     &=  -{1 \over 2\pi |M|}\int _{\SS^1}\left( \int_{\partial \phi^{-1}(\cdot)} \iota_u\iota_v\mu \right)\sigma
     = -{1 \over 2\pi |M|}\int _{\SS^1}\sigma \int_\gamma \iota_u\iota_v \mu \\
     & = {1 \over  |M|} \int_\gamma \iota_v\iota_u \mu
     =  \omega^{\rm MW}_\gamma(X^u_\gamma, X^v_\gamma)
\end{align}
where $\gamma=\Pi\psi$ and $X^u, X^v$ are the fundamental vector fields of $u,v\in\Gamma(TM)$ on $\cO=\Pi\cF$ (\autoref{def:FundamentalVectorFieldX}).
 Finally \autoref{prop:dPi} asserts  $d\Pi_\psi Y^{(u,a)}=X^u_{\Pi\psi}$ and $d\Pi_\psi Y^{(v,b)}=X^v_{\Pi\psi}$, showing $d\Theta_\psi(Y^{(u,a)}, Y^{(v,b)})=\omega_{\Pi\psi }^{\rm MW} (d\Pi_\psi Y^{(u,a)}, d\Pi_\psi Y^{(v,b)})$. This concludes the proof.
\end{proof}

 We have now shown that \(\Th\) is a formal prequantum potential for $\omega_{\rm{MW}}$. We conclude this section by noting that the formal prequantization \(\Pi\colon(\cF_\cO,\Th)\to (\cO,\omega)\) is not a prequantum bundle. This is because we cannot factor \(\Th\) onto \(\cO\), as the kernel of \(d\Pi\)  is not contained within the kernel of \(\Th\). For example, for \(\dot \psi = e^{ic_1} \psi\), \(\ring \psi = e^{ic_2} \psi \in T_\psi \cF_\cO\) with some real constants \(c_1 \neq  c_2\), we have \(d\Pi(\dot \psi) = d\Pi(\ring \psi) = 0 \in T_{\Pi\psi } \cO\), but the values \(\Th(\dot \psi) = \frac{  c_1}{2\pi}\) and \(\Th(\ring \psi) =  \frac{ c_2}{2\pi}\) are different.


However, it is still possible to take a quotient space of \(\cF_\cO\) where the prequantum form can descend onto. This construction actually leads to a prequantum bundle structure, as we will explain in the next section.

\section{Prequantum structure}\label{sec:prequantum_bundle}
We have shown that the fiberation $\Pi:(\cF_\cO,\Theta)\to(\cO,\omega^{\rm MW})$ is a formal prequantization \ie, $d\Theta=\Pi^*\omega^{\rm MW}$ (\autoref{th:quantum_potetial}). 
Building on this, we now construct a prequantum bundle as a quotient bundle of $\cF_\cO$, equipped with a connection form whose curvature form recovers the MW symplectic form on $\cO$.

Recall \autoref{def:GeneralizedPrequantumBundle} that a generalized prequantum bundle over a symplectic manifold \((B,\beta)\), where \(G = \SS^1\times H\) for some discrete Abelian group \(H\), is a principal \(G\)-bundle \(\pi\colon E\to B\) with a connection form \(\alpha\in\Omega^1(E)\) satisfying \(\pi^*\beta = d\alpha\).  As a principal bundle with connection, a generalized prequantum bundle comes with the notion of vertical and horizontal distributions.



\begin{definition}[Vertical and horizontal distributions, and horizontal lift]\label{def:unique_horizontal_lift}
    Let $\pi\colon (E,\alpha)\to (B,\beta)$ be a generalized prequantum bundle. At each point $x \in E$, the tangent space splits as \( T_x E = V_x E \oplus H_x E \), where the vertical distribution is defined by \( V_x E = \ker d\pi|_x \), and the horizontal distribution is given by \( H_x E = \ker \alpha|_x \). A \emph{horizontal lift} of a path \( \{ \gamma_t \}_t \subset B \) is a path \( \{ \ell_t \}_t \subset E \) satisfying \( \pi \circ \ell_t = \gamma_t \) and \( \partial_t \ell_t \in H_{\ell_t} E \) for all \( t \),
    which is unique up to the choice of the initial point \(\ell_0\).
\end{definition}

The formal prequantization  \( \Pi\colon (\cF_\cO, \Th) \to (\cO, \omega^{\rm MW}) \) is not a generalized prequantum bundle since the fibers are not homeomorphic to a one-dimensional group \(G\), and \( \Th \) is not a connection form.
In particular, the supposed horizontal subspace \(\ker\Th\) and the vertical subspace \(\ker d\Pi\) do not split \(T\cF_\cO\).
To define a bundle on which \( \Th \) becomes a genuine connection 1-form, we take the quotient of the tangent space at each \( \psi \in  \cF_\cO \) by the intersection \( \ker \Th \cap \ker d\Pi \). This quotient process is characterized by the following equivalence relation.

\subsection{Volume Bundle}

\begin{definition}[Volume class]
     Two elements $\psi_0,\psi_1\in \cF_\cO$ are \emph{volumetrically equivalent}, denoted by $ \psi_0\sim_\cP \psi_1$, if there exists a path $\{\psi_t\}_{t\in[0,1]} $ 
connecting $\psi_0$ and $\psi_1$ such that $\partial_t\psi_t\in \ker d\Pi_{\psi_t} \cap \ker\Theta_{\psi_t}$ for all $t\in[0,1]$. Each equivalence class $[\psi]_\cP$ is called \emph{a volume class}.
\end{definition}

Note that if \(\psi_0\sim_\cP\psi_1\), then they necessarily lie in the same fiber (\ie, \(\Pi\psi_0 = \Pi\psi_1\)) since they are connected by a path tangent to \(\ker d\Pi\).  In particular, \(\Pi\) descends to a well-defined map on the quotient space \(\cF_\cO/\sim_\cP\).

\begin{definition}[Volume bundle]
    Define the \emph{volume bundle} as the quotient space $\cP\coloneqq \cF_\cO / \sim_\cP$.  The projection \(\Pi_\cP\colon \cP\to\cO\) is defined as the map induced by \(\Pi\colon \cF_\cO\to\cO\).
\end{definition}

Note that the projection $\Pi \colon \cF_\cO \to \cO$ decomposes into two projections: the projection associated to the quotient $\pi_\cP \colon \cF_\cO \to \cP$, and $\Pi_\cP \colon \cP \to \cO$.

The tangent space \(T_{[\psi]_\cP}\cP\) of the volume bundle at each $[\psi]_\cP\in\cP$ is given by $T_{[\psi]_\cP}\cP=T_\psi \cF_\cO/\sim_\cP$ where the induced equivalence relation \(\sim_\cP\) on each tangent space is given by
\begin{align}\label{eq:SimPOnTangent}
    \dot\psi\sim_\cP \ring\psi\quad\text{if and only if}\quad\dot\psi-\ring\psi\in \ker d\Pi|_\psi\cap \ker \Th|_\psi.
\end{align}
This characterization of the tangent space involving \(\ker\Theta|_{\psi}\) implies that the 1-form \(\Th\in\Omega^1(\cF)\) descends to a well-defined 1-form \(\Th_\cP\in\Omega^1(\cP)\), with \(\Th = \pi_\cP^*\Th_\cP\).

Since $d\Th = \Pi^* \omega^{\rm MW}$, we also have $d\Th_\cP = \Pi_\cP^* \omega^{\rm MW}$. Therefore, the volume bundle $\Pi_\cP \colon (\cP, \Th_\cP) \to (\cO, \omega^{\rm MW})$ is a formal prequantization.

In addition, the volume bundle $\cP$ forms a principal $G$-bundle over $\cO$, where the structure group is $G = \SS^1 \times H^1(M; \ZZ)$, acting on $\cP$ as follows. First, define a circle action by constant phase shifts: 
\begin{align}
\label{eq:S1ActionOnF}
    e^{2\pi\ii c} \rhd [\psi]_\cP \coloneqq [e^{2\pi\ii c}\psi ]_\cP, \quad \text{for } \psi \in \cF, \; e^{2\pi\ii c} \in \SS^1.
\end{align}
This action is well-defined: For \(\psi_0\sim_{\cP}\psi_1\) connected by a path \(\{\psi_t\}_{t\in[0,1]}\) tangential to \(\ker d\Pi\cap\ker\Theta\), the shifted path \(\{e^{2\pi\ii c}\psi_t\}_{t\in[0,1]}\) is also tangential to \(\ker d\Pi\cap\ker\Theta\), implying that \([e^{2\pi\ii c}\psi_0]_\cP=[e^{2\pi\ii c}\psi_1]_{\cP}\).
The associated fundamental vector field of the action \eqref{eq:S1ActionOnF} is given by \(Z\colon\RR\to \Gamma(T\cP)\)
\begin{align}
    Z^c|_{[\psi]_\cP} = [2\pi\ii c\psi]_{\cP},\quad c\in\RR, [\psi]_{\cP}\in\cP,
\end{align}
where the equivalence class on the right-hand side refers to \eqref{eq:SimPOnTangent}.
Here, we identify the Lie algebra \(T_1\SS^1 = 2\pi\ii\RR\) with the real line \(\RR\) by the identification \(c\mapsto 2\pi\ii c\).  In particular, \(Z^1\in\Gamma(T\cP)\) is the \emph{unit vertical vector field} arising from the circle action.

\begin{proposition}\label{prop:ThetaOfZIsOne}
    \(\Th_\cP(Z^1)|_{[\psi]_{\cP}}=1\) for all \([\psi]_{\cP}\in\cP\).
\end{proposition}
\begin{proof}
    \(\Th_\cP(Z^1)|_{[\psi]_{\cP}} = \Th_\psi(2\pi\ii\psi) \mathrel{\overset{\text{\autoref{lem:value_of_Theta}}}{=}} {1\over 2\pi |M|}\int_{M}\Im(2\pi\ii)\mu = 1\).
\end{proof}

\begin{proposition}\label{prop:S1ActionFreeTransitive}
    The \(\SS^1\)-action \eqref{eq:S1ActionOnF} is free and transitive on each connected component of each fiber of \(\Pi_{\cP}\colon\cP\to\cO\).
    Moreover, \(\psi_0,\psi_1\in\Pi^{-1}\gamma\subset\cF\) lie in the same connected components of the fiber \(\Pi^{-1}\gamma\) if and only if \([\psi_0]_\cP, [\psi_1]_\cP\) lie in the same connected components of the fiber \(\Pi_\cP^{-1}\gamma\subset\cP\).
\end{proposition}
\begin{proof}
    The free-action property follows from the fact that the unit fundamental vector field \(Z^1\) is nowhere vanishing, ensured by \autoref{prop:ThetaOfZIsOne}.
    
    The remainder of the proof shows the transitivity of the action and the statement about connectivity of each fiber of \(\cP\).
    
    We first show that for any \(\psi_0, \psi_1 \in \Pi^{-1}\gamma\subset\cF\) in the same connected component of \(\Pi^{-1}\gamma\), \([\psi_1]_\cP\) is in the orbit of \([\psi_0]_\cP\) by the \(\SS^1\) action. Let $\{\psi_t\}_{t\in[0,1]}$ be a path joining $\psi_0$ and $\psi_1$. Define another path $\{\tilde\psi_t\}_{t\in[0,1]}$ by $\tilde\psi_t\coloneqq e^{-2\pi \ii c(t)} \psi_t  $ with $c(t)\coloneqq \int^t_0 \frac{d}{dt'}\Theta_{\psi_t'}(\partial_{t'} \psi_{t'}) dt'$. A direct computation shows $\partial_t \tilde\psi_t\in \ker \Theta_{\tilde \psi_t}$ for any $t\in[0,1]$, and hence $\psi_0\sim \tilde \psi_1$. On the other hand $\psi_1 = e^{2\pi c(1) \ii} \tilde \psi_1 $. Therefore, $e^{2\pi c(1) \ii}\rhd [\psi_0]_\cP= e^{2\pi c(1) \ii}\rhd [\tilde \psi_1]_\cP=[\psi_1]_\cP$.

    Therefore, the \(\SS^1\)-action is free and transitive on the projection of each connected component of the fiber of \(\cF_\cO\) in \(\cP\).  
    Thus, each of these projection images of fiber components of \(\cF_\cO\) is a circle in a fiber of \(\cP\).  
    These circles must be disjoint:  If \([\psi_0]_\cP, [\psi_1]_\cP\) lie in the same circle in the same fiber of \(\cP\), then there exists \(e^{2\pi c\ii}\in\SS^1\) so that \(\psi_1 = e^{2\pi c\ii}\psi_0\), which implies that \(\psi_0,\psi_1\) lies in the same \(\DC\)-orbit and thus in the same fiber component of \(\cF_\cO\).  Hence, each fiber \(\Pi_\cP^{-1}\gamma\) of \(\cP\) must consist of a disjoint union of circles, each of which is descended from a connected component of \(\Pi^{-1}\gamma\subset \cF_\cO\), and is freely and transitively acted by the \(\SS^1\) action \eqref{eq:S1ActionOnF}.
\end{proof}

Recall that each fiber \(\Pi^{-1}\gamma\) consists of connected components indexed by the discrete group \(H^1(M;\ZZ)\) (\autoref{prop:H1Z action on fiber}).  This and \autoref{prop:S1ActionFreeTransitive} implies that each fiber \(\Pi_\cP^{-1}\gamma\) of \(\cP\) also consists of connected components indexed by the same group \(H^1(M;\ZZ)\).  

We may further specify a group action of \( H^1(M; \ZZ) \) on \( \cP \) as a deck transformation between the connected components of each fiber using the map $\mathfrak h$ \eqref{eq:frakhDefinition}. Observe that $\mathfrak h$ induces an isomorphism
\begin{align}
    \frak h^*\colon C^\infty(M;\CC^\times)/\on{Exp}(C^\infty(M;\CC))\to H^1(M;\ZZ)
\end{align}
since $\frak h (e^{\varphi})(C)={1\over 2\pi\ii}\oint_{C}\frac{d e^\varphi}{e^\varphi} ={1\over 2\pi\ii}\oint_{C} d\varphi=0$ for any closed curve $C$ on $M$ due to the Stokes theorem.

Fix a basis $[\eta_1],\ldots,[\eta_\beta]$ for \(H^1(M;\ZZ)\), where $\beta$ is the first Betti number of $M$.  Then choose arbitrary representatives $\tau_i\in C^\infty(M;\CC^\times)$ of ${\frak h^*}^{-1}[\eta_i]$.
Using these functions \(\tau_i\), we define the action of $H^1(M;\ZZ)$ by
\begin{align}
    [\eta] \rhd [\psi]_\cP \coloneqq [\tau_1^{n_1}\cdots\tau_\beta^{n_\beta} \cdot \psi]_\cP, \quad \psi \in \cF, \ [\eta]= [n_1 \eta_1+\ldots + n_\beta \eta_\beta]\in H^1(M;\ZZ).
\end{align}

These $\SS^1$- and $H^1(M;\ZZ)$-actions together define a $G=\SS^1\times H^1(M;\ZZ)$-action
\begin{align}\label{eq:G_action}
    (e^{2\pi\ii c}, [\eta])\rhd [\psi]_\cP = [e^{2\pi\ii c} \tau_1^{n_1}\cdots\tau_\beta^{n_\beta}\psi]_\cP, \quad e^{2\pi\ii c}\in \SS^1, [\eta]\in H^1(M;\ZZ),
\end{align}
which is free, transitive, and fiber preserving. We thus obtain the following result.
\begin{proposition}\label{prop:principal_G_bundle}
The fibration $\Pi_\cP\colon \cP\to\cO$ equipped with the above action of $G=\SS^1\times H^1(M;\ZZ)$ is a principal $G$-bundle. 
\end{proposition}

Furthermore, $\Th_\cP$ defines a connection form on this principal $G$-bundle:
\begin{proposition}\label{prop:connection_form}
     The 1-form $\Th_\cP$  is a connection form on the principal $G$-bundle $\Pi_\cP\colon\cP\to \cO$. That is, $\Th_\cP$ satisfies the following two properties:
\begin{enumerate}
    \item Equivariance under the $G$-action $\Phi_g\colon \cP\to \cP$; \ie, $\Phi_g^*\Th_\cP=\Th_\cP$ for every $g\in G$.
    \item Vertical reproducibility; \ie, $\Th_\cP(Z^c)=c$ for any $c\in\RR\cong \frak{g}$ and its fundamental vector field $Z^c\in \Gamma(T\cP)$.%

\end{enumerate}
\end{proposition}

\begin{proof}
The vertical reproducibility follows from \autoref{prop:ThetaOfZIsOne}.  
The equivariance can be verified by the following straightforward calculation.
Let $g=(e^{2\pi\ii c},[\sum_i n_i\eta_i])\in G=\SS^1\times H^1(M;\ZZ)$, and denote by $\Phi_g$ the $g$-action \eqref{eq:G_action}. Then,
\begin{align}
    \Th_\cP|_{\Phi_g([\psi]_\cP)}([\Phi_{g*}\dot\psi]_\cP)
    &={1\over 2\pi|M|}\int_M \Im \frac{\dot \psi e^{2\pi\ii c}\tau_1^{n_1}\cdots \tau_\beta^{n_\beta} \cdot \overline{\psi e^{2\pi\ii c}\tau_1^{n_1}\cdots \tau_n^{n_\beta}}}{|\psi e^{2\pi\ii c}\tau_1^{n_1}\cdots \tau_\beta^{n_\beta}|^2} \ \mu 
    \\&= {1\over 2\pi|M|} \int_M \Im \frac{\dot \psi  \bar\psi }{|\psi |^2} \ \mu 
    =\Th_\cP|_{[\psi]_\cP}([\dot \psi]_\cP),
\end{align}
where $\psi\in \cF$ and $\dot \psi\in T_\psi \cF$ are any representatives of $[\psi]_\cP$ and $[\dot\psi]_\cP$ respectively.
\end{proof}

Combining the results that $\Pi_\cP\colon(\cP, \Th_\cP)\to(\cO,\omega^{\rm MW})$ is a formal prequantization (\autoref{th:quantum_potetial}) and a principal $G$-bundle with $G=\SS^1 \times H^1(M;\ZZ)$ (\autoref{prop:principal_G_bundle}), equipped with a connection form $\Th_\cP$ (\autoref{prop:connection_form}), we obtain our main theorem:

\begin{theorem}\label{th:generalized_prequantum_bundle}
The volume bundle $\Pi_\cP\colon (\cP,\Th_\cP)\to (\cO, \omega^{\rm MW})$ is a generalized prequantum $G$-bundle with $G = \SS^1 \times H^1(M;\ZZ)$. 
\end{theorem}

\begin{corollary}\label{cor:prequantum_circle_bundle}
    For simply-connected \(M\), the volumne bundle is a prequantum circle bundle.
\end{corollary}

\subsection{Interpretation of the Volume Bundle Prequantization}

The prequantization \(\Pi_{\cP}(\cP,\Theta_\cP)\to(\cO,\omega^{\rm MW})\) for the MW structure allows for geometric interpretation of the MW form as the curvature of a bundle.  The evaluation of the curvature can be given in terms of the holonomy of horizontal lifts. 

Using the connection form $\Theta_\cP$, each tangent space $T_{[\psi]_\cP}\cP$  splits into the vertical distribution $V_{[\psi]_\cP}\cP=\ker d\Pi_\cP|_{[\psi]_\cP}$ and the horizontal distribution $H_{[\psi]_\cP}\cP=\ker\Th_\cP|_{[\psi]_\cP}$, and one can consider the unique horizontal lift on $\cP$ of a path in $\cO$ (\autoref{def:unique_horizontal_lift}). 


In light of \autoref{rmk:geometric_meaning_of_Th}, this horizontal lift can be interpreted as the evolution of the implicit representation $[\psi]_\cP$ such that the average swept volume of phase level hypersurfaces remains zero at all times.
That is, for a representative $\psi_t$ of a horizontal lift $[\psi_t]_\cP$, we have for the family of phase hypersurfaces \(\{\sigma_t^s\}_{s\in \SS^1}\) defined by $\sigma_t^s = \phi_t^{-1}(s)$,
\begin{align}
    \int_{\SS^1}\int_{\sigma_t^s}\iota_{\partial_t{\sigma}_t^s} \mu \ ds = 0
\end{align}
at each $t$ where $\partial_t \sigma_t^s$ is the velocity defined on each hypersurface $\sigma_t^s$. 

This reveals a geometric interpretation of the MW form as the curvature form of $\Th_\cP$, measuring the holonomy induced by parallel transport on $\cP$ over a closed path in $\cO$:

\begin{corollary}[Average swept volume]\label{cor:average_swept_volume}
Consider a closed path \(\partial\Sigma\) in \(\cO\) that bounds a 2-dimensional disk \(\Sigma\), representing a cyclic motion of a codimension-2 submanifold \(\gamma_t \subset M\) for \(0 \leq t \leq 1\), with \(\gamma_0 = \gamma_1\). Let $[\psi_t]_\cP$ be a horizontal lift in \(\cP\) of $\gamma_t$ and $\{\psi_t\}_t$ be representatives taken continuously in $t$. Suppose 
that $\psi_0$ and $\psi_1$ lie in the same connected component of a fiber.

Then the volume enclosed between phase hypersurface \(\sigma_0^s=\phi^{-1}_0(s)\) and \(\sigma_1^s=\phi^{-1}_1(s)\), averaged over \(s\in \SS^1\), equals to \(|M|\iint_{\Sigma} \omega^{\rm MW}\).
\end{corollary}

As illustrated in \autoref{fig:DC_orbit_U}, we may compress the phase field so that it is non-constant only within narrow bands.
By considering a limiting case of \autoref{cor:average_swept_volume}, where  $\arg \psi_t\coloneqq M\setminus\psi^{-1}(0)\to \RR/2\pi \ZZ$ becomes constant except a \(2\pi\) jump on a single hypersurface \(\sigma_t\) bounding \(\gamma_t\), we obtain the following:

\begin{corollary}[Swept volume by a hypersurface] \label{cor:single_surface_swept_volume}
Let \(\Sigma \subset \cO\) and \(\{\gamma_t\}_{t \in [0,1]} \subset M\) be as in \autoref{cor:average_swept_volume}. Suppose that each \(\gamma_t\) bounds a hypersurface, i.e., \(\gamma_t = \partial \sigma_t\), and that the  volume swept out by \(\sigma_t\) remains zero at each \(t\), meaning
\(\int_{\sigma_t}\iota_{\partial_t {\sigma}_t} \mu = 0\).
Then, the volume enclosed between \(\sigma_0\) and \(\sigma_1\) is given by \(|M|\iint_{\Sigma} \omega^{\rm MW}\).
\end{corollary}
Note that the interpretation of the MW form in \autoref{cor:single_surface_swept_volume} reduces to the swept volume of a single surface, no longer explicitly involving the complex function.
\begin{remark}
   \autoref{cor:single_surface_swept_volume} can also be shown directly in the framework for the explicit shape space (\autoref{sec:explicit symplectic potential}). Let us take an \(m{-}1\) dimensional submanifold \(\Sigma_0\) of \(M\) bounded by some \(\gamma_0 \in \cO\), and consider the orbit \(\cS \coloneqq \Diff_0(M) \rhd \Sigma_0\), where the action is defined by \(f \rhd \Sigma = f \circ \Sigma\) for \(f \in \Diff_0(M)\) and \(\Sigma \in \cS\). Then \(\pi \colon \cS \to \cO\) is a fiber bundle where the fibration is given by the boundary operator, i.e., \(\pi(\Sigma) = \partial \Sigma\), and the tangent space at each \(\Sigma\) is \(T_\Sigma \cS = \{u \circ \Sigma \mid u \in \diff(M)\}\).

   Define a 1-form \(\eta\) on \(\cS\) by \(\eta_\Sigma(u \circ \Sigma) = \int_\Sigma \iota_u\mu\), which measures the flux of $u$ through $\Sigma$. Then \(\eta\) serves as a formal prequantization, i.e., \(d\eta = \pi^*\omega^{\rm MW}\), which can be shown in a manner similar to the proof of \autoref{thm:SymplecticPotentialForExactVolumeForm}. For a path \(\{\gamma_t\} \subset \cO\), there exist infinitely many lifts $\{\ell_t\}\subset \cS$ such that $\partial_t\ell_t\in\ker \eta$ for all $t$, but the notion of no swept volume still makes sense, and we recover \autoref{cor:single_surface_swept_volume}.
\end{remark}

\bibliographystyle{amsalpha}
\bibliography{reference}
\end{document}